\documentclass[11pt]{article}
\usepackage{fullpage}
\usepackage{graphicx} 
\usepackage{amsfonts}
\usepackage{amsthm}
\usepackage{amsmath}
\usepackage{amssymb}
\usepackage{mathabx}
\usepackage{mathrsfs}
\usepackage{mathptmx}
\usepackage{comment}
\usepackage{xcolor}
\usepackage[
backend=biber,
style=alphabetic,
sorting=ynt,
doi=false,
url=false,
isbn=false
]{biblatex}
\usepackage{titling}
\usepackage{hyperref}

\newcommand{\Sym}{\mathop{\mathrm{Sym}}}

\newcommand{\PP}{\mathbb{P}}
\newcommand{\EE}{\mathbb{E}}
\newcommand{\Unif}{\mathrm{Unif}}
\newcommand{\Binom}{\mathrm{Binom}}
\newcommand{\uppersupent}{\mathop{\overline{H}^{\mathrm{superlin}}}}
\newcommand{\lowersupent}{\mathop{\underline{H}^{\mathrm{superlin}}}}
\newcommand{\supent}{\mathop{H^{\mathrm{superlin}}}}
\newcommand{\upperlinent}{\mathop{\overline{H}^{\mathrm{lin}}}}
\newcommand{\lowerlinent}{\mathop{\underline{H}^{\mathrm{lin}}}}
\newcommand{\linent}{\mathop{H^{\mathrm{lin}}}}

\newtheorem{theorem}{Theorem}[section]
\newtheorem{lemma}[theorem]{Lemma}
\newtheorem{prop}[theorem]{Proposition}
\newtheorem{corollary}[theorem]{Corollary}

\newtheorem{claim}[theorem]{Claim}
\theoremstyle{definition}
\newtheorem{defi}[theorem]{Definition}
\newtheorem{megj}[theorem]{Remark}
\newtheorem{quest}[theorem]{Question}
\usepackage{titling}

\title{On the sampling entropy of permutons}
\thanksmarkseries{alph}
\author{Bal\'azs Maga\thanks{HUN-REN Alfréd Rényi Institute of Mathematics, Budapest, Hungary. Email: \texttt{magab@renyi.hu} \newline Supported by the KKP 139502 project, funded by the Ministry of Innovation and Technology of Hungary from the National Research, Development and Innovation Fund.}}
\date{}

\begin{document}

\maketitle
\begin{abstract}
For a permuton $\mu$ let $H_n(\mu)$ denote the Shannon entropy of the sampling distribution of $\mu$ on $n$ points. We investigate the asymptotic growth of $H_n(\mu)$ for a wide class of permutons. 

We prove that if $\mu$ has a non-vanishing absolutely continuous part, then $H_n(\mu)$ has a growth rate $\Theta(n \log n)$. We show that if $\mu$ is the graph of a piecewise continuously differentiable, measure-preserving function $f$, then $H_n(\mu)/n$ tends to the Kolmogorov--Sinai entropy of $f$. Using genericity arguments, we also prove the existence of function permutons for which $H_n(\mu)$ does not converge either after normalizing by $n$ or by $n\log n$.

We study the sampling entropy of a natural family of random fractal-like permutons determined by a sequence of i.i.d. choices. It turns out that for every $n$, $H_n(\mu)/n$ is heavily concentrated. We prove that the sequence $H_n(\mu)/n$ either converges or has deterministic log-periodic oscillations almost surely, and argue towards the conjecture that in nondegenerate case, oscillation holds. On the other hand, for a straightforward random perturbation of the model $\tilde{\mu}$ of $\mu$, we prove the almost sure convergence of $H_n(\tilde{\mu})/n$.

\medskip

\noindent \textbf{Keywords}: permutons, Shannon entropy, Kolmogorov--Sinai entropy, $d$-ary tree, log-peridocity, genericity
\end{abstract}

\section{Introduction}

A permuton is a probability measure on the unit square with uniform marginals. Permutons are the natural limit objects of permutations \cite{HOPPEN201393}, where convergence is defined via convergent substructure densities arising from sampling (cf. graphons and graph convergence). Permutons enjoyed much attention in the recent years, see for example \cite{ALON2022102361}, \cite{BASSINO2022108513}, \cite{GLEBOV2015112}.

A permuton $\mu$ is naturally approximated by finite permutations sampled from it. The sampling method of order $n$ is that we draw $n$ i.i.d. points according to $\mu$ and read their height ordering from left to right, obtaining a random permutation $\pi_n$. We then have  $\pi_n\to\mu$ almost surely as $n\to\infty$ \cite[Lemma~4.2]{HOPPEN201393}. A very natural question concerns how much information is carried by this sampling, i.e., what is the Shannon entropy of $\pi_n$? We call the resulting sequence the {\it sampling entropy sequence} of $\mu$, denoted by $H_n(\mu)$, and investigate its asymptotic properties in this paper.

\subsection{The main results}

The first natural question concerns with the growth rate of $H_n$. For each $n$, $H_n(\mu)$ is maximized by the Lebesgue measure $\mu$, as in that case the distribution of $n$-patterns is uniform over $\Sym(n)$. Thus in this case, $H_n=\log {n!}\sim n \log n$. This exhibits that the sampling entropy sequence might have a superlinear growth rate, which turns out to be the general behaviour of permutons with non-vanishing absolutely continuous part:

\begin{theorem} \label{thm:abs_cont}
    If $\mu$ is absolutely continuous with respect to the Lebesgue measure, then
    $$\lim_{n\to\infty} \frac{H_n(\mu)}{n\log n}=1.$$
    More specifically, if the absolutely continuous part of $\mu$ is denoted by $\mu_{\mathrm{ac}}$, then
    $$\liminf_{n\to\infty} \frac{H_n(\mu)}{n\log n}\geq\|\mu_{\mathrm{ac}}\|.$$
\end{theorem}

We continue our investigation by studying a notable class of singular permutons, i.e., permutons defined by measure-preserving maps.
If $f:[0,1]\to[0, 1]$ is a measure-preserving map, we can associate with it a permuton $\mu_f$ via $\mu_f(A\times B) = \lambda(A\cap f^{-1}(B))$. The simplest nontrivial example is given by the doubling map, $f(x)=\{2x\}$, for which one can deduce that $H_n(\mu_f)\sim n\log 2$, predicting that for permutons of this form, $H_n(\mu)$ grows linearly . We prove the following theorem:

\begin{theorem}\label{thm:differentiable_maps}
If $f$ is measure-preserving and piecewise continuously differentiable with finitely many pieces, then
$$\lim_{n\to\infty}\frac{H_n(\mu)}{n}=\int_{0}^{1}\log|f'|.$$
\end{theorem}

Recalling Rokhlin's formula (\cite{Coudene2016}[Theorem~12.1]), this yields the following remarkable result:
\begin{corollary}\label{cor:kolmogorov_sinai_entropy}
    If $f$ is measure-preserving and piecewise continuously differentiable with finitely many pieces and bounded derivative, then
    $$\lim_{n\to\infty}\frac{H_n(\mu)}{n})=h_{KS}(f),$$
    where $h_{KS}(f)$ is the Kolmogorov--Sinai entropy of $f$.
\end{corollary}

We note the existence of examples consisting of infinitely many continuously differentiable pieces for which Theorem \ref{thm:differentiable_maps} fails (constructed in \cite[Theorem~4.6]{TDK}), but since Rokhlin's formula is still valid in this case, Corollary \ref{cor:kolmogorov_sinai_entropy} fails in general.

Theorems \ref{thm:abs_cont}-\ref{thm:differentiable_maps} might hint that $H_n(\mu)/n$ and $H_n(\mu)/n\log n$ should always have a limit, though they might be 0 or $+\infty$. It turns out to be a good idea to look for counterexamples using a Baire category argument in a well-chosen space. To this end, we consider the space $C(\lambda)$ of continuous measure-preserving functions, extensively studied by \cite{Bobok_2020}, \cite{Bobok_2022}. We prove the following theorem. 

\begin{theorem}\label{thm:generic_lack_of_entropy}
For the generic $f\in C(\lambda)$, we have
$$\lim_{n\to\infty}\frac{H_n(\mu)}{n}=0, \ \lim_{n\to\infty}\frac{H_n(\mu)}{n\log n}=1.$$
\end{theorem}

Finally, we study how the sampling entropy sequence behaves for following two natural families of random permutons. For the first, we fix some $d$ and a distribution $F\in\PP(\Sym(d))$, and pick a random permutation $\pi_1$ according to this. We consider the permuton $\mu_1$ naturally associated with it, supported by $d^{-1}\times d^{-1}$ grid squares ordered according to $\pi_1$, being uniform restricted to each of these. Call the replacement of the Lebesgue measure by $\mu_1$ a {\it refinement step}, and then define the sequence $\mu_n$ inductively by carrying out totally independent refinement steps in each of the $d^{-(n-1)}\times d^{-(n-1)}$ grid squares supporting $\mu_{n-1}$. Then $\mu_n$ corresponds to a permutation $\pi_n$ in general, and the weak limit of the sequence $\mu_n$ (or $\pi_n$) is the random permuton $\mu_F$ we study, induced by an almost everywhere defined measure-preserving bijection with graph $\bigcap_{n=0}^{\infty}\mathrm{supp}(\mu_n)$, well-defined a.e. 
\begin{figure}[h]
\centering
\includegraphics[width=12cm]{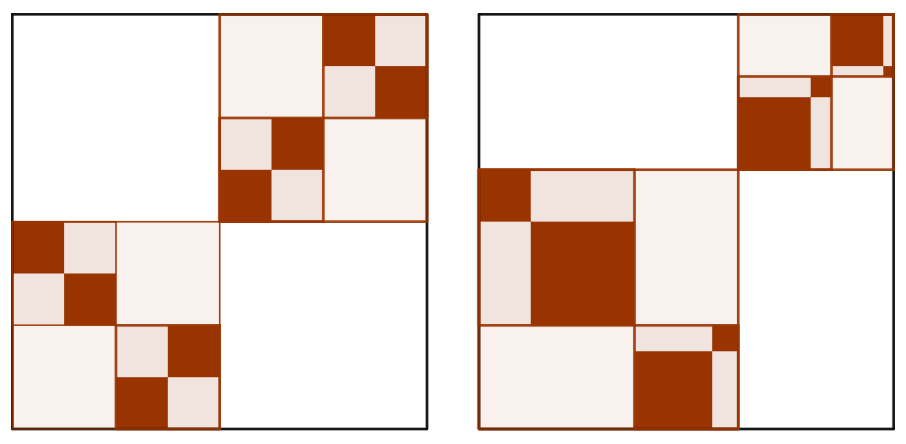}
\caption{(Left) The third approximation $\mu_3$ of an outcome of $\mu_F$ for $d=2$. The filled squares form the support, restricted to any of these $\mu_3$ is uniform. Shaded squares visualize $\mu_1, \mu_2$.  (Right) The third approximation of $\mu_{F, \Unif}$, where the random permutations are the same as on the left, but the sidelengths vary.}
\label{fig:random_permuton}
\end{figure}

As randomness in this construction comes from a sequence of i.i.d. choices, one expects $H_n(\mu)/n$ being almost surely convergent by being heavily concentrated for every $n$. We could only prove something weaker, i.e., that the linearly normalized sampling entropy sequence is asymptotically log-periodic (the precise definition is given later):

\begin{theorem}\label{thm:random_automorphism}
    For the random permuton $\mu=\mu_F$ induced by $F\in \PP(\Sym(d))$, $H_n(\mu)/n$ is asymptotically log-periodic in base $d$ with limit
    \begin{equation} \label{eq:aut_thm_statement}
    A(x)=\sum_{r=-\infty}^{\infty}\left(\sum_{l=1}^{\infty}\rho_{l+1}\frac{1-d^{-l}}{(l+1)! \log d}\Gamma\left(l-\frac{2\pi i r}{\log d}\right)\right)\exp (2\pi i r x),
    \end{equation}
    for some nonnegative sequence $\rho_l=O(\log l)$ determined by $F$. The average $\int_{0}^{1}A(x)dx$ vanishes if and only if
    $F$ is a Dirac measure on $(1,2\dots ,d)$ or $(d,d-1,\dots ,1)$. 
\end{theorem}

The second family we study is the natural random perturbation of $\mu_F$, defined using the following extra source of randomness: instead of using the $d^{-1}\times d^{-1}$ grid squares in the first refining step, we determine the sidelengths by drawing $d-1$ points independently and uniformly at random from the $x$-axis and adjust the squares to the resulting partition (see Figure \ref{fig:random_permuton}). We denote this permuton by $\mu_{F, \Unif}$, with a more formal definition given in the later sections. Given the difficulties arising in the study of the first setup, we find the next result miraculous:

\begin{theorem}\label{thm:random_automorphism2}
    There exists some deterministic $c\in [0, +\infty)$ such that for the random permuton $\mu=\mu_{F, \Unif}$ induced by $F\in \PP(\Sym(d))$, $H_n(\mu)/n\to c$ almost surely. Moreover, $c=0$ if and only if $F$ is a Dirac measure on $(1,2\dots ,d)$ or $(d,d-1,\dots ,1)$.
\end{theorem}

The marvelous feature of Theorem \ref{thm:random_automorphism2} is that it means that for almost every choice of the sidelengths, $H_n(\mu)/n$ converges to a fixed value, while Theorem \ref{thm:random_automorphism} displays that 
actually verifying this for a given fiber is highly nontrivial (the ones resulting in nonequipartitions seem even more cumbersome to approach). This theorem also highlights that the convergence of $H_n(\mu)/n$ is not exclusive to graphs of functions with good regularity, but can emerge from other convenient structural properties.

We are highly interested in the answer to the following questions:

\begin{quest}
    Can one fix a fiber of the random sidelength choices in the definition of $\mu_{F, \Unif}$ such that for the resulting random permuton $\mu$, $H_n(\mu)/n$ converges?

    Does $H_n(\mu_F)/n$ converges in Theorem \ref{thm:random_automorphism}?
\end{quest}

\subsection{Insights and further speculation}

First we note that the entropy of an absolutely continuous permuton with density $g$ have been defined in \cite{Kenyon2019} as its differential entropy $-\int_{[0, 1]^2} g\log g$. They used it in their study of permutons constrained by having fixed densities of a finite number of patterns, and showed that the limit shapes are determined by maximizing differential entropy over permutons with those constraints. However, this notion of entropy is not directly related to the sampling entropy sequence. While the former's natural habitat is by definition the class of absolutely continuous permutons, the latter exhibits quite interesting behavior for some natural families of singular permutons, as our theorems demonstrate.

The idea behind the aforementioned lack of equivalence between the Kolmogorov--Sinai entropy and the limit of $H_n(\mu)/n$ is that if $[0,1]$ decomposes into a sequence of disjoint invariant subsets over which $f$ is dynamically trivial ($|f'|=1$ except for a finite set of discontinuities), then the Kolmogorov--Sinai entropy vanishes, but if $\mu_f$ provides an ever-improving approximation over these invariant subsets of the Lebesgue measure, then $H_n(\mu)$ can be large for every $n$. More fundamentally, as hinted by Theorem \ref{thm:abs_cont}, sampling entropy can skyrocket due to very small perturbations as well, while Kolmogorov--Sinai entropy cannot under mild conditions. Due to this phenomenon, exploring the exact scope of Corollary \ref{cor:kolmogorov_sinai_entropy} would be welcome.

Another remarkable aspect of Corollary \ref{cor:kolmogorov_sinai_entropy} comes from the comparison to \cite{BandtKellerPompe}, where permutation entropy for interval maps is defined by considering the Shannon entropy of the random pattern determined by plotting the point sequence $(f^k(x), k)_{k=0}^{n-1}$ for the partial forward orbit of the uniform random point $x$ and reading the height ordering from the left to the right. For this concept, the same conclusion is achieved, i.e., that permutation entropy equals the Kolmogorov--Sinai entropy under mild conditions. The question of sufficient and necessary conditions for this equivalence is open as well, with a recent generalization given by \cite{GutjahrKeller}. Developing a better understanding of the relationship of these three approaches to entropy seems to be an exciting challenge.

\begin{quest}
    What are the limitations of Theorem \ref{thm:differentiable_maps} and Corollary \ref{cor:kolmogorov_sinai_entropy}? 

    Can a direct link be provided between our permuton entropy notion and permutation entropy in the sense of \cite{BandtKellerPompe}, without using Kolmogorov--Sinai entropy as a bridge?
\end{quest}

The following natural question interesting on its own right frequently recurs during our analysis: given some planar probability measures $\mu_i$ ($i=1, 2, \dots, m$) and a convex combination $\mu=\sum_{i=1}^{m}\alpha_i\mu_i$, what is the relationship between $H_n(\mu)$ and the $H_n(\mu_i)$s? In general, it is hard to say much: considering for example the two permutons concentrated on the diagonal and the antidiagonal, respectively, they both have vanishing entropy, yet their nontrivial convex combinations have linearly growing entropy, as a simple (but nontrivial) argument shows. This essentially falls under the scope of the technical Lemma \ref{lemma:entropy_of_convex_comb} unnecessary to be stated precisely now that in general, $H_n(\mu)$ can be expressed by the $H_k(\mu_i)$ ($k=0, \dots, n$) with at most linear error. However, as proved in Lemma \ref{lemma:geom_sep}, under some geometric conditions concerning the supports of the pre-permutons, this error collapses to have logarithmic size, giving much stronger control. The existence of such a formula means that if $\mu$ has a self-similar structure, then up to some accumulating error, the sequence $H_n(\mu)$ satisfies a recurrence relation, with the coefficients being compatible with $H_n(\mu)\sim Cn$. These observations will be vital in our study of the random permuton families we introduced.

Another way to describe the way we create $\pi_n$ in the construction of $\mu_F$ is via the substitution operation: for permutations $\sigma, \alpha_1, ..., \alpha_n$, the permutation $\sigma[\alpha_1, ..., \alpha_n]$ consists of $n$ consecutive subsequences isomorphic to $\alpha_1, ..., \alpha_n$ whose relative order is the same as the relative order of the terms of $\sigma$. With this notation, $\pi_n=\pi_{n-1}[\alpha_1, \dots, \alpha_{d^{n-1}}]$, where $\alpha_i\in \Sym(d)$ are drawn independently according to $F$. Note that $\mu_F$ can also be identified with a random automorphism of the rooted infinite $d$-ary tree $T(d)$. Indeed, such automorphisms are in one-to-one correspondence with $\Sym(d)$-labelling of the vertex set, describing how branches going out from the given vertex are permuted. Specifying the permutations in the level $n$ vertices (with the root being at level 0) corresponds to specifying $\pi_{n+1}$ in the construction above. From this perspective, our permuton construction corresponds to the random automorphism of the $d$-ary tree where branches from any vertex permuted with law given by $F$.

By a simple observation, the patterns of the random permutons $\mu_{F}$ and $\mu_{F,\Unif}$ form a pattern-avoiding permutation class $X$ which is substitution-closed almost surely, i.e., for $\sigma, \alpha_1, ..., \alpha_n\in X$, $\sigma[\alpha_1, ..., \alpha_n]\in X$ holds as well. This concept was introduced by \cite{ALBERT20051}, an important example is given by the class of separable permutations, i.e., permutations avoiding the patterns 3142 and 2413. (Here and later as well we always use the one-line notation, but for better readability, we sometimes deploy parentheses and commas.) Probably the most well-known permuton-related result about such classes was given by \cite{BrownianSep}, where it is shown that taking the uniform random permutation from a substitution-closed class subject to mild conditions, the limit is a random permuton coming from a one-dimensional family of deformations of the Brownian separable permuton, originally introduced as the limit of the uniform random separable permutations. 

We note that $\mu_F$ and $\mu_{F,\Unif}$ being pattern-avoiding immediately yields that their entropy grows at most linearly by the celebrated theorem of Marcus and Tardos, proving the famous conjecture of Richard Stanley and Herbert Wilf proposed around 1992 that the number of permutations avoiding a pattern $\sigma$ is $s_n(\sigma)\sim c(\sigma)^n$ (\cite{MARCUS2004153}, for a short summary of related conjectures, we also refer to this paper). This result can in fact be interpreted that pattern avoidance imposes a linear bound on the entropy, however, simple constructions exhibit that this is not reversible \cite[Theorem~4.1]{TDK}.

By asymptotic log-periodicity in Theorem \ref{thm:random_automorphism}, we mean the following:

\begin{defi}\label{def:log_periodic}
    We say that a sequence $a(n)$ is {\it asymptotically log-periodic} with base $\eta$ if the sequence of functions $A^{(m)}$ defined by
    $$A^{(m)}(x)=a([\eta^{x+m}])$$
    is locally uniformly convergent in $\mathbb{R}$. The limit $A(x)$ is called the {\it log-periodic limit} of $a(n)$ in base $\eta$.
\end{defi}

As we will see, the increments $H_{n+1}(\mu) -H_n(\mu)$ are of $O(log(n))$, which quickly yields that log-periodicity proven above is roughly the best amount of regularity one can hope for in the lack of convergence. If the log-periodic limit exists, it must be periodic by 1, thus we could have defined the log-periodic limit via uniform convergence in $[0, 1]$. The sequence $a(n)$ is convergent if and only if its asymptotic log-periodic limit is constant, which can be read from its Fourier expansion. Thus $H_n(\mu_F)=Cn+o(n)$ if and only if the Fourier coefficients in \eqref{eq:aut_thm_statement} vanish for $r\neq 0$. Seemingly there is no fundamental reason why this should happen, fueling the conjecture that it is not the case, however, dealing with this question seems to be very difficult. Solely comparing the magnitude of the terms in the series
$$\sum_{l=1}^{\infty}\rho_{l+1}\frac{1-q^l}{(l+1)! |\log q|}\Gamma\left(l+\frac{2\pi i r}{\log q}\right)$$
is insufficient to decide if it vanishes, thus a precise knowledge of their exact value seems to be necessary, which is difficult not primarily because of the hardly accessible $\Gamma$ terms, but because of the hardly describable and computable $\rho$s. As we will see, these coefficients involving the $\Gamma$ function come from the study of some surprisingly nonconvergent self-averaging sequences. The same setup is studied in \cite{ESS1993}, a more general setup is considered and the corresponding literature about nonconvergent self-averaging sequences is overviewed to some extent in \cite{self-averaging}. The key takeaway of that paper is that if a sequence is defined by a recurrence where each term is some convex combination of earlier terms, or in a probabilistic language, $a(n) = \EE[a(X_n)]$ for a random variable with values in $[n-1]$, then if $\EE X_n \approx \beta n$ and $\mathrm{Var} X_n\ll n^2$ (i.e. the fluctuations are too small to evenly cover $[n-1]$), then it is natural to expect that $a(n)$ does not converge, contrary to the other extreme when $X_n$ is uniform. Studying the entropy of $\mu_F$ leads to the study of the self-averaging sequence where $X_n$ is essentially $\Binom(n, 1/d)$, concentrated too heavily around $n/d$ to derail convergence. Introducing the random perturbation giving rise to $\mu_{F, \Unif}$ smoothes out this adversarial effect, yielding self-averaging sequences which actually converge.

We note that log-periodic fluctuations also arise as error terms in analysis of algorithms from similar calculations \cite{logperiodicalgorithms}, though the customarily used definition in that literature is that a sequence $a(n)$ is asymptotically log-periodic in base $\eta$ if $a(n)=A(\log_\eta n)+o(1)$ for some function $A$ periodic by 1. We use the stronger Definition \ref{def:log_periodic} as it implies the uniqueness of the limit function, making the discussion slightly more convenient, while the latter allows some ambiguity if $A$ is allowed to be discontinuous.

Finally, we remark some loose parallels with the findings of \cite{DHM2015}. Notably, $H_n(\mu)/n$ possibly admitting quite wild oscillations (Theorem \ref{thm:generic_lack_of_entropy}) somewhat resembles \cite{DHM2015}[Proposition~2.1], in which the size of the largest clique of subgraphs sampled from graphons are studied, and it turns out that this quantity can have essentially arbitrarily bad oscillations asymptotically almost surely. Moreover, 
Theorem \ref{thm:random_automorphism} resonates with \cite{DHM2015}[Theorem~2.2]: while the largest clique size can oscillate, for fixed $n$ it is heavily concentrated. This is a loose parallel as the source of randomness is different: in case of \cite{DHM2015}[Theorem~2.2], it comes from the sampling, while in our case, it is encoded in the permuton being random.

\subsection{The organization of the paper}

Before starting the bulk of discussion, in Section \ref{sec:prelim}, some fundamental results about permutons and entropy are recalled, and also some, later quite important technical lemmas are proved concerning the sampling entropy sequence, such as basic bounds on its increments and its interaction with taking convex combinations.

Section \ref{sec:abs_cont} is devoted to the proof of Theorem \ref{thm:abs_cont}, carried out with ease building upon the findings of Section \ref{sec:prelim}.

In Section \ref{sec:kolmogorov_sinai}, we prove Theorem \ref{thm:differentiable_maps}.

In Section \ref{sec:generic}, we prove Theorem \ref{thm:generic_lack_of_entropy}.

In Section \ref{sec:random}, after investigating the almost sure pattern structure of the random permuton family introduced in the previous subsection and invoking some results about self-averaging sequences, the proof of Theorem \ref{thm:random_automorphism} is presented.

In Section \ref{sec:random2}, building upon the progress in Section \ref{sec:random}, we prove Theorem \ref{thm:random_automorphism2}.

\section{Preliminaries} \label{sec:prelim}

\subsection{Permutons} \label{subsec:prelim_perm}

To get to the limit theory of permutations, we first need a notion of sampling and substructure densities. Given a permutation $\pi$ of $[n]$, that is $\pi\in \Sym(n)$, its restriction to any $k$-element subset $I$ of $[n]$ gives rise to a permutation in $\Sym(k)$, i.e., the one to which it is order isomorphic. Iterating over the possible $k$-element subsets, the relative frequency of $\sigma\in \Sym(k)$ denoted by $t(\sigma, \pi)$ is the density of $\sigma$ in $\pi$. These densities form a probability distribution on $\Sym(k)$.

Permutons as limits of permutations were introduced in \cite{HOPPEN201393}. A sequence $(\pi_n)$ of permutations is convergent if $t(\sigma, \pi_n)$ converges for any pattern $\sigma$. If $|\pi_n|\to\infty$, with any such sequence one can associate a permuton, that is a probability measure in the unit square with uniform marginals, and any permuton arises as the limit of some convergent sequence of permutations. Pattern densities can also be defined in permutons: for a permuton $\mu$ and $\sigma\in \Sym(n)$, the density $t(\sigma, \mu)$ is defined as the probability of seeing the permutation $\sigma$ upon sampling $n$ points from $\mu$ and reading their height ordering from left to right. Denote the resulting distribution on $\Sym(n)$ by $\mu^{(n)}$. This sequence is far from being arbitrary: it is immediate that it must satisfy a system of compatibility conditions, stemming from the identity
\begin{equation} \label{eq:compatibility_conditions}
t(\sigma, \mu) = \sum_{\pi\in \Sym(n)}t(\sigma, \pi)t(\pi, \mu) 
\end{equation}
holding for any $k\leq n$, $\sigma\in \Sym(k)$, $\pi\in \Sym(n)$. As displayed by \cite{Král2012}, much more sophisticated statements can be made, i.e., if $\mu^{(4)}$ is uniform, then $\mu$ is the Lebesgue measure, and hence any $\mu^{(k)}$ is uniform. 

Slightly extending our scope, the sequence $\mu^{(n)}$ can be defined for any planar probability measure $\mu$ with atomless marginals, ruling out collisions. Such measures are called pre-permutons in \cite{Dubach2023LocallyUR}, this more general point of view will prove to be useful.

We say that a permuton $\mu$ avoids the pattern $\sigma$ if $t(\sigma, \mu)=0$. Otherwise $\mu$ contains $\sigma$. Patterns contained by a permuton trivially form a permutation class, i.e., a family of permutations closed under taking subpermutations. Such a class can be defined by giving a set of excluded patterns, which can be finite or infinite. Studying pattern avoidance in the context of permutons has been the subject of \cite{Garbe2024}, which provided a structure theorem, stating that almost every fiber of the disintegration of a $\sigma$-avoiding permuton consists of at most $|\sigma|$ atoms, and applied this theorem to give a soft analysis proof of a removal lemma. 

By appropriately rescaling its permutation matrix, with a permutation $\pi$ one can naturally associate an absolutely continuous permuton. More specifically, the associated permuton $\mu_\pi$ has density function
\[
g(x, y) =
\begin{cases}
    n\text{ if } (x, y)\in \left[\frac{i-1}{n}, \frac{i}{n}\right] \times \left[\frac{\pi(i)-1}{n}, \frac{\pi(i)}{n}\right]\text{ for some $i\in[n]$},\\
    0\text{ otherwise.}
\end{cases}
\]
We call the $1/n$ grid squares in the support of $g$ the \textit{base squares} of $\pi$ or $\mu_\pi$.

As proven in \cite{HOPPEN201393}, a sequence of permutations $(\pi_n)$ converges to $\mu$ if and only if $\mu_{\pi_n}$ tends to $\mu$ weakly. Moreover, in the space of permutons the following equivalence of topologies is proven:

\begin{prop}[{{\cite[Lemma~5.3]{HOPPEN201393}}}]\label{prop:equiv_topologies}
    For the permutons $\mu, (\mu_n)_{n=1}^{\infty}$, the following are equivalent:
    \begin{itemize}
        \item $\mu_n\to \mu$ weakly.
        \item $t(\pi, \mu_n)\to t(\pi, \mu)$ for any pattern $\pi$.
        \item $d_{\square}(\mu_n, \mu)\to 0$, where $d_{\square}(\nu_1, \nu_2) = \sup_{R=[x_1, x_2]\times [y_1, y_2]\subseteq[0,1]^2} |\nu_1(R)-\nu_2(R)|$.
    \end{itemize}
\end{prop}

Using the second form of this topology, it is immediate that $H_n$ is a continuous functional for any $n$.

\subsection{Basic properties of entropy}

Let $\psi(x)=-x\log x$. The following estimates are essentially the standard bounds on conditional entropy, written in the form of entropy bounds for convex combination of probability measures:
\begin{lemma}\label{lemma:basic_entropy_estimates}
Let $(\nu_k)_{k=1}^{m}$ be probability measures on the same finite ground set $E$ and $(\alpha_k)_{k=1}^{m}$ a nonnegative vector with $\ell_1$ norm $\|\alpha\|_1$. Then
$$\sum_{k=1}^{m}{\alpha_k}H(\nu_k)\leq \|\alpha\|_1 H\left(\frac{\sum_{k=1}^{m}{\alpha_k}\nu_k}{\|\alpha\|_1}\right)\leq \sum_{k=1}^{m}{\alpha_k}H(\nu_k) +\|\alpha\|_1 H\left(\frac{(\alpha_k)_{k=1}^m}{\|\alpha\|_1}\right)$$
\end{lemma}

We use the notation $\lowersupent(\mu) = \liminf \frac{H_n(\mu)}{n\log n}$ and $\uppersupent(\mu) = \limsup \frac{H_n(\mu)}{n\log n}$ for what we call lower and upper superlinear entropy. If these coincide, we say that $\mu$ has superlinear entropy, denoted by $\supent(\mu)$. In a similar virtue we speak of lower and upper linear entropy ($\lowerlinent(\mu)$, $\upperlinent(\mu)$), and if they coincide, we say that $\mu$ has linear entropy, denoted by $\linent(\mu)$. 

We also take note of the fact that for any probability measure $\mu$ and $\pi\in\Sym(k)$, we can express $t(\pi, \mu)$ as a $k$-fold integral over a suitable set: if
$$T_{\pi} = \{(x_1, y_1), ..., (x_k, y_k)\ \text{s.t. they determine pattern}\ \pi\}\subseteq \mathbb{R}^{2k}$$
then
$$t(\pi, \mu)=\int_{T_{\pi}}d\mu(x_1, y_1)...d\mu(x_k, y_k).$$
This has the following handy corollary, due to the change of measures formula:

\begin{lemma} \label{lemma:abs_cont_bound}
    Assume $\mu_1 \ll \mu_2$, and that for $g=\frac{d\mu_1}{d\mu_2}$, we have $1/K<g<K$. Then
    $$K^{-k} t(\pi, \mu_2) \leq t(\pi, \mu_1) \leq K^k t(\pi, \mu_2).$$
    Consequently,
    $$H_k(\mu_2) -  k \log K \leq H_k(\mu_1) \leq H_k(\mu_2) + k \log K,$$
    and
    $$\lowersupent(\mu_1)=\lowersupent(\mu_2), \ \uppersupent(\mu_1)=\uppersupent(\mu_2)$$
\end{lemma}

The monotonicity of $H_k$ seems to be a subtle problem: while the previous results follow from quite general observations about entropy, monotonicity should not as the sampling entropy sequence of a permutation is not monotone. In a related observation, note that $H_k$ cannot be submodular or supermodular in general due to comparing Fekete's lemma and Theorem \ref{thm:generic_lack_of_entropy}. Nevertheless we conjecture that monotonicity holds, more specifically, for fixed $k$ and a long enough permutation $\pi$ of length $n>N_k$, we have $H_k(\pi)\leq H_{k+1}(\pi)$, but proving this seems to require a deeper understanding of pattern density sequences. However, a quasimonotonicity result is easy to prove and provides sufficient regularity for certain applications:

\begin{lemma} \label{lemma:entropy_quasi_mon}
For any pre-permuton $\mu$
    $$H_{k-1}(\mu)-\log k \leq H_k(\mu) \leq H_{k-1}(\mu)+\log k.$$
\end{lemma}

\begin{proof}
    For the first inequality take $\pi_k\in \Sym(k)$ according to $\mu^{(k)}$, and for any index set $I\in \binom{[k]}{k-1}$ define $\pi_k^{I}$ as the subpermutation determined by the entries in $I$. Then $\pi_k\mapsto\pi_k^{I}$ is a deterministic mapping, hence $H(\pi_k^{I})\leq H(\pi_k)$, and for the distributions $\mu^{(k),I}$ of $\pi_k^{I}$, we have 
    $$\frac{1}{k}\sum_{I\in\binom{[k]}{k-1}}\mu^{(k),I}=\mu^{(k-1)}.$$
    Thus by Lemma \ref{lemma:basic_entropy_estimates}, we see
    $$H(\mu_{k-1})- H((1/k)_{i=1}^{k})\leq \frac{1}{k}\sum_{I\in\binom{[k]}{k-1}}H(\mu^{(k),I})\leq H(\mu^{(k)}).$$
    Plugging in $H((1/k)_{i=1}^{k}) = \log k$ concludes the first part.

    For the second part observe that by \eqref{eq:compatibility_conditions} there is a matrix $M\in \mathbb{R}^{\Sym(k)\times \Sym(k-1)}$ mapping the probability vector $\mu^{(k)}$ to $\mu^{(k-1)}$ in which each column sums up to 1 and each row sums up to $k$. Let $M_*\in \mathbb{R}^{\Sym(k)\times \Sym(k)}$ consist of $k$ identical copies of $M/k$ stacked vertically. Then $M_*$ is doubly stochastic, hence by Lemma \ref{lemma:basic_entropy_estimates},
    $$H_k(\mu) = H(\mu^{(k)})\leq H(M_*\mu^{(k)}) = H_{k-1}(\mu) +\log k,$$
    the second equality following from the fact that $M_*\mu^{(k)}$ is just $k$ copies of $\mu^{(k-1)}/k$.
\end{proof}

Lemma \ref{lemma:entropy_quasi_mon} is also significant as it shows that $H_n(\mu)/n$ cannot produce too quick nontrivial oscillations.

\subsection{Permutons from automorphisms of the $d$-ary tree} \label{subsection:random_aut_notation}

In this subsection, we introduce some notation and define formally the permutons we study in Theorems \ref{thm:random_automorphism}-\ref{thm:random_automorphism2}. First, note that the nodes of the infinite rooted $d$-ary tree $T_d$ can be encoded by the set $\Omega=\{0, 1, \dots, d-1\}^{<\omega}$ of finite $d$-ary strings. Then a labelling $S=(\alpha_s)_{s\in\Omega}$ of the nodes by length $d$ permutations (i.e., $S\in \Sym(d)^{\Omega}$) encodes a sequence of permutations nested in the sense of substitution: we can recursively define a sequence of permutations with $$\pi_1=\alpha_{\emptyset}, \quad \pi_n= \pi_{n-1}[(\alpha_s)_{|s|=n}]\text{ for $n>1$},$$
where $(\sigma_s)_{|s|=n}$ is enumerated lexicographically. Geometrically, it just means that each of the base squares of the permuton $\mu_{n}=\mu_{n, S}$ corresponding to $\pi_n$ is replaced by the appropriately scaled base squares of the subsequent $\sigma_s$. Observe that each base square of any $\mu_n$ directly corresponds to a node of $T_d$. The sequence $(\mu_n)$ tends weakly to the permuton $\mu=\mu_S$ induced by the almost everywhere defined measure-preserving bijection $f$ whose graph is $\bigcap_{n=0}^{\infty}\mathrm{supp}(\mu_n)$, 
well-defined a.e.

We will need to vary the sidelengths to define the permuton arising in Theorem \ref{thm:random_automorphism2}. We will encode this by introducing edge weights $t_s\in [0, 1]$ for $s\in {\Omega\setminus\{\emptyset\}}$ conditioned on
$$\sum_{i\in \{0, 1, \dots, d-1\}}t_{s|_{n-1}\oplus i}=1,$$
i.e., the weights on edges incident to a node and pointing away from the root sum to 1. (For any $s\in\Omega$ and $k\leq |s|$, $s|_{k}\in\Omega$ denotes the substring formed by the first $k$ entries of $s$) The permuton $\mu=\mu_{S, t}$ corresponding to vertex labels $S$ and edge weights $t$ is defined as follows: its $n$th approximation $\mu_{n, S, t}$ is focused on subsquares of $[0, 1]^2$ in the same relative order as $\mu_{n, S}$ in the simple vertex-labelled setup, but the sidelength of the square corresponding to a vertex $v\in \Omega$ is defined by $T(s)=\prod_{\emptyset\neq s\subseteq v}t_s$, with the subset relation standing for $s$ being an initial slice of $v$. Again, it is simple to see that this uniquely determines a permuton, and $(\mu_{n, S, t})$ tends weakly to a permuton $\mu=\mu_{S, t}$, which is induced by the almost everywhere defined measure-preserving bijection once $T(s)\to 0$ as $|s|\to\infty$. The first construction corresponds to $t \equiv 1/d$.

We will consider two distinct random choices definable via this formalism. In the first setup, the only source of randomness is that the node labels are chosen i.i.d. according to a distribution $F$ on $\Sym(d)$, while $t\equiv 1/d$. In the second setup, the edge weights are also drawn at random, independently from the vertex labels the following way: for any $s\in \Omega$, we sample $U_1, \dots, U_{d-1}\sim \mathrm{Unif}(0, 1)$ independently, denote their order statistics by $U_{(1)}, \dots, U_{(n)}$ and let the random variables $\tau_{s\oplus i}=U_{(i)}-U_{(i-1)}$, where $U_{(0)}=0$ and $U_{(d)}=1$ define the edge weights. We denote the resulting random permutons in the two setups by $\mu_{F}$ and $\mu_{F, \mathrm{Unif}}$, a fiber of the latter with given edge weights is denoted by $\mu_{F, \tau=t}$. It is natural to think of $\mu_{F, \mathrm{Unif}}$ as a random perturbation of $\mu_{F}$.

\subsection{Entropy of convex combinations of pre-permutons}

We conclude the preliminaries by proving two technical lemmas, concerning how the entropy of a convex combination of pre-permutons is related to the individual entropies. As noted in the introduction, these provide strong motivation for our choices of the random permutons to study.

The first of these lemmas gives a general bound concerning the superlinear entropy. This will be used to great effect: most of the work needed to prove Theorem \ref{thm:abs_cont} will be done in the proof of this lemma. Note that these bounds can be made slightly sharper via a more careful analysis of the liminfs and the limsups, but it is not of our direct interest currently.

\begin{lemma}\label{lemma:entropy_of_convex_comb}
Let $\mu = \sum_{i=1}^m\beta_i\mu_i$ for $\beta_i\in[0, 1]$, $\sum_{i=1}^{m}\beta_i=1$, where the pre-permutons $(\mu_i)_{i=1}^m$ are pairwise singular. Then for arbitrary $l\in [m]$,
$$\beta_l \lowersupent(\mu_l)\leq \lowersupent(\mu)\leq \beta_l\lowersupent{\mu_l} + \sum_{i=1 ,i\neq l}^{m}\beta_i\uppersupent(\mu_i),$$
$$\beta_l \uppersupent(\mu_l)\leq \uppersupent(\mu)\leq \sum_{i=1}^{m}\beta_i\uppersupent(\mu_i).$$ 
\end{lemma}

\begin{proof}
    Sample $n$ points from $\mu$. The idea of the proof is that if pairwise disjoint sets $(A_i)_{i=1}^m$ are chosen such that $\mu_i(A_j)=\delta_{i, j}$, then roughly $\beta_i n$ points fall into $A_i$, and the sampling entropy measured by $H_n(\mu)$ comes from 2 sources: the pattern entropies of the $\mu_i$s admitted on these numbers of points, and the entropy of the interlacement of these patterns. It will turn out that the second source is negligible after normalizing by $n\log n$, while the terms coming from the first source have contribution $\sum_{i=1}^m H_{\beta_i n}(\mu_i)$ asymptotically.

    Denote by $(x_1, y_1), ..., (x_n, y_n)$ the $n$ points sampled from $\mu$ such that $x_1< ... < x_n$, and let $\pi$ be pattern determined by these points. Denote by $I_j$ the set of indices for which $(x_i, y_i)\in A_j$. Then the number of possible choices for the sequences $(I_j)_{j=1}^m$, $(\pi(I_j))_{j=1}^m$ is bounded by $m^{2n}$. Denote by $E_{\mathbf{I}, \mathbf{J}} = E_{(I_1, ..., I_m), (J_1, ..., J_m)}$ the event that we see a given pair of sequences of index sets, by $\alpha_{\mathbf{I}, \mathbf{J}}$ its probability, and by $\mu^{(n), \mathbf{I}, \mathbf{J}}$ the corresponding conditional pattern distribution. Then 
    $$\mu^{(n)} = \sum_{\mathbf{I}, \mathbf{J}}\alpha_{\mathbf{I}, \mathbf{J}}\mu^{(n), \mathbf{I}, \mathbf{J}},$$
    and thus by Lemma \ref{lemma:basic_entropy_estimates} and the bound on the number of possible pairs $(\mathbf{I}, \mathbf{J})$

    \begin{equation} \label{eq:abs_cont_lem_entropy_bound}
    \begin{split}
    \sum_{\mathbf{I}, \mathbf{J}}\alpha_{\mathbf{I}, \mathbf{J}}H(\mu^{(n),\mathbf{I}, \mathbf{J}})\leq H(\mu^{(n)})=H\left(\sum_{\mathbf{I}, \mathbf{J}}\alpha_{\mathbf{I}, \mathbf{J}}\mu^{(n),\mathbf{I}, \mathbf{J}}\right)&\leq H((\alpha_{\mathbf{I}, \mathbf{J}}))+\sum_{\mathbf{I}, \mathbf{J}}\alpha_{\mathbf{I}, \mathbf{J}}H(\mu^{(n),\mathbf{I}, \mathbf{J}}) \\ &\leq 2n\log m + \sum_{\mathbf{I}, \mathbf{J}}\alpha_{\mathbf{I}, \mathbf{J}}H(\mu^{(n),\mathbf{I}, \mathbf{J}}).
    \end{split}
    \end{equation}
    Conditioned on any of the events $E_{\mathbf{I}, \mathbf{J}}$, the pattern $\pi$ is uniquely determined by the subpatterns in the $A_j$s. However, we have no independence, thus $H(\mu^{(n),I,J})$ is not simply given, but bounded from above by the sum of entropies of the subpatterns. Introducing the notation $|\mathbf{I}|=(|I_1|, ..., |I_m|)$, and only summing for a given size sequence of 
    $$\mathbf{k}=(k_1, ..., k_m),$$ 
    we may write
    \begin{equation}\label{eq:convexity_lemma_calc1}
    \sum_{|\mathbf{I}]=|\mathbf{J}|=\mathbf{k}}\alpha_{\mathbf{I}, \mathbf{J}}H(\mu^{(n),\mathbf{I}, \mathbf{J}}))\leq 
    \sum_{|\mathbf{I}]=|\mathbf{J}|=\mathbf{k}}\alpha_{\mathbf{I}, \mathbf{J}}\sum_{i=1}^m H(\mu_{A_i}^{(n),\mathbf{I}, \mathbf{J}}).
    \end{equation}
    where $\mu_{A_i}^{(n), \mathbf{I}, \mathbf{J}}\in \mathbb{P}(\Sym(k_i))$ is the distribution of patterns determined by points in $A_i$ conditioned on the event $E_{\mathbf{I}, \mathbf{J}}$. A lower bound can be obtained by using just one of the measures, for arbitrary $l\in [m]$ we have
    \begin{equation}\label{eq:convexity_lemma_calc2}
    \sum_{|\mathbf{I}]=|\mathbf{J}|=\mathbf{k}}\alpha_{\mathbf{I}, \mathbf{J}}H(\mu^{(n),\mathbf{I}, \mathbf{J}})\geq \sum_{|\mathbf{I}]=|\mathbf{J}|=\mathbf{k}}\alpha_{\mathbf{I}, \mathbf{J}} H(\mu_{A_l}^{(n),\mathbf{I}, \mathbf{J}}).
    \end{equation}
    By Lemma \ref{lemma:basic_entropy_estimates}, both estimates provided by \eqref{eq:convexity_lemma_calc1} and \eqref{eq:convexity_lemma_calc2} can be continued by taking the convex combination inside the entropy, in the case of the lower bound at price of the entropy of the normalized vector $\left(\alpha_{\mathbf{I}, \mathbf{J}}\right)_{|\mathbf{I}]=|\mathbf{J}|=\mathbf{k}}$, which is at most $2n \log m$. Moreover, with $k=k_l$, we clearly have
    $$\sum_{|\mathbf{I}]=|\mathbf{J}|=\mathbf{k}}\alpha_{\mathbf{I}, \mathbf{J}}\mu_{A_l}^{(n),\mathbf{I}, \mathbf{J}}=\mu_l^{(k)}\sum_{|\mathbf{I}]=|\mathbf{J}|=\mathbf{k}}\alpha_{\mathbf{I}, \mathbf{J}},$$
    as the left hand side is proportional to the total pattern distribution of $k$ points falling in $A_l$, which is $\mu_l^{(k)}$ by definition. 
    Plug these findings into \eqref{eq:abs_cont_lem_entropy_bound}, and first, focus on finishing the lower bound. To simplify notation, let $\alpha_{\mathbf{k}}^l = \sum_{|I_l|=k_l}\alpha_{\mathbf{I}, \mathbf{J}}$.
    Then we have
    \begin{equation} \label{eq:abs_cont_lem_entropy_bound_simp}
        H(\mu^{(n)})\geq -2n\log m + \sum_{k=0}^n \alpha_{\mathbf{k}}^lH(\mu_l^{(k)}).
    \end{equation}
    Here $\alpha_{\mathbf{k}}^l$ is the probability mass of the binomial distribution $\mathrm{Binom}(n, \beta_l)$. Thus by quasimonotonicity (Lemma \ref{lemma:entropy_quasi_mon}), for any $\varepsilon>0$, and large enough $n$, with $n_{-}=\left[(1-\varepsilon)\beta_l n\right]$, the final sum can be bounded from below to get
    $$H(\mu^{(n)}\geq-2n\log m + (1-\varepsilon)H(\mu_l^{(n_{-})})-2\varepsilon n\log n.$$
    Dividing by $n\log n$ and taking liminf in $n$,
    $$\lowersupent(\mu)\geq (1-\varepsilon)\liminf_n \frac{H(\mu_l^{(n_{-})})}{n_{-}\log n_{-}} \frac{n_{-}\log n_{-}}{n\log n}-2\varepsilon. $$
    This liminf is just $(1-\varepsilon)\beta_l\lowersupent{\mu_l}$, thus as this holds for any $\varepsilon>0$ and $l\in[m]$, we obtain the desired lower bound for the lower superlinear entropy. The proof of the lower bound on upper superlinear entropy goes verbatim. 

    Now consider the upper bounds, use the simplified notation
    $\alpha_{\mathbf{k}} = \sum_{|\mathbf{I}|=\mathbf{k}}\alpha_{\mathbf{I}, \mathbf{J}}$. 
    In this case, by the arguments above instead of \eqref{eq:abs_cont_lem_entropy_bound_simp}, we obtain the bound
    \begin{displaymath}
        H(\mu^{(n)})\leq 2n\log m + \sum_{\mathbf{k}} \alpha_{\mathbf{k}}\sum_{i=1}^m H(\mu_i^{(k_i)}).
    \end{displaymath}
    Here $\alpha_k$ is the probability distribution of the multinomial distribution $\mathrm{Multinom}(n, \mathbf{\beta})$. Moreover, any of the entropies on the right hand side can be bounded by $n\log n (1+o(1))$. Thus for any $\varepsilon>0$, and large enough $n$, we have
    $$\sum_{\mathbf{k}} \alpha_{\mathbf{k}}\sum_{i=1}^m H(\mu_i^{(k_i)}) =
    O(\varepsilon n\log n) + \sum_{\mathbf{k}\in n\mathbf{\beta}(1\pm\varepsilon)}\alpha_{\mathbf{k}}\sum_{i=1}^m H(\mu_i^{(k_i)}).$$
    Again by quasimonotonicity, for $n_{i+}\sim \beta_i(1+\varepsilon)n$,
    at the price of an error at most $2\varepsilon n\log n$ any index $k_i$ can be replaced by $n_{i+}$. If we replace all such indices, the total error is still $O(\varepsilon n\log n)$, thus we find
    \begin{equation} \label{eq:concise_entropy_convex_comb}
    \sum_{\mathbf{k}} \alpha_{\mathbf{k}}\sum_{i=1}^m H(\mu_i^{(k_i)}) = O(\varepsilon n\log n) + \sum_{i=1}^{m}H(\mu_i^{(n_{i+})}).
    \end{equation}
    This is hence an upper bound on $H(\mu^{(n)})$. If we divide by $n\log n$, the limsup of this sum is bounded from above by the sum of the limsups of the individual terms, and
    $$\limsup\frac{H(\mu_i^{(n_{i+})})}{n}=\beta_i(1+\varepsilon)\uppersupent(\mu_i),$$
    similarly to the calculation in the lower bound. As $\varepsilon\to 0$, this precisely gives the intended upper bound for $\uppersupent(\mu)$.
    
    The upper bound for $\lowersupent(\mu)$ follows from the same argument, repeatedly using the inequality $\liminf (a_n + b_n)\leq \liminf a_n + \limsup b_n.$
\end{proof}

Our second lemma is based on the observation that if we know a bit more about the geometric structure of the supports of the $\mu_i$s, a much more precise conclusion can be drawn about the relationship of the entropy of $\mu$ and the entropies of the $\mu_i$s. We say that the planar probability measures $\mu, \nu$ are \textit{geometrically separated} if both of their marginals are ordered, i.e., if we sample $(x_\mu, y_\mu)\sim \mu$ and $(x_\nu, y_\nu)\sim \nu$, then $x_\mu\leq x_\nu$ a.s. or $x_\mu\geq x_\nu$ a.s., and $y_\mu\leq y_\nu$ a.s. or $y_\mu\geq y_\nu$ a.s. The essence of the next lemma is that if the $\mu_i$s are pairwise geometrically separated, then the entropy sequence of $\mu$ is determined solely by the particular entropy sequences, up to negligible error, as no interlacement is possible.

\begin{lemma} \label{lemma:geom_sep}
    Let $\mu = \sum_{i=1}^m\beta_i\mu_i$ for $\beta_i\in[0, 1]$, $\sum \beta_i=1$, where the pre-permutons $(\mu_i)_{i=1}^m$ are pairwise geometrically separated. Then for any $n$,
    $$H(\mu^{(n)})= \omega_{n, m} + \sum_{i=1}^{m}\sum_{k=1}^{n} \binom{n}{k}\beta_i^{k}(1-\beta_i)^{n-k}H(\mu_i^{(k)}),$$
    with the error term satisfying $0\leq \omega_{n, m} \leq m \log(n+m)$
\end{lemma}

\begin{megj}
    Observe that if we drop the error term, the sequences defined by $H(\mu^{(n)}), H(\mu_i^{(n)})=Cn$ satisfy the same identity. This leads to the idea discussed in the introduction that in case of a self-similar decomposition, when $H(\mu^{(n)})= H(\mu_i^{(n)})$, $\mu$ should have linear entropy.
\end{megj}

\begin{megj}
    This error term is very far from being exact if $n\ll m$, however, we will use it with $n\gg m$, in which case it is sufficiently good. We opted not to state this result in unnecessary generality at the expense of readability. In some applications, $m$ will be fixed along the argument, so we will simply write $O(\log n)$, but other applications will involve $n, m$ jointly growing, requiring this form in which the asymptotics in $m$ being tracked as well.

    This error term arises in the proof as a simple upper bound for the entropy of a multinomial distribution, for which no exact formulae are available, hence its precise handling is quite a challenge. For some recent results, see \cite{Kaji}.
\end{megj}

\begin{proof}[Proof of Lemma \ref{lemma:geom_sep}]
    Sample $n$ points from $\mu$, and choose $(A_i)_{i=1}^{m}$ as in the proof of Lemma \ref{lemma:entropy_of_convex_comb}, with the extra property that they witness geometric separatedness. Let $\pi$ be the random permutation determined by these points.
    
    Denote by $E_{k_1, ...k_m}$ the event that $k_i$ points fall into $A_i$, and let $\pi_i$ be the random permutation determined by these points. Conditioned on any of these events, the distribution of $\pi$ is just the independent coupling of $(\pi_1, ..., \pi_m)$ due to separatedness. Consequently, if $\mu^{(n, k_1, ..., k_m)}\in\mathbb{P}(\Sym(n))$ is the measure obtained by conditioning $\pi$ on $E_{k_1, ...k_m}$, then
    $$H(\mu^{(n, k_1, ..., k_m)})=\sum_{i=1}^{m}H(\mu_i^{(k_i)}).$$
    Since 
    $$\mu^{(n)}=\sum_{k_1+ \dots + k_m=n}\mathbb{P}(E_{k_1, ...k_m})\mu^{(n, k_1, ..., k_m)},$$
    by Lemma \ref{lemma:basic_entropy_estimates}
    $$\sum_{k_1+ \dots + k_m=n}\mathbb{P}(E_{k_1, ...k_m})\sum_{i=1}^{m}H(\mu_i^{(k_i)})\leq H(\mu^{(n)}) \leq H((k_1, ..., k_m)) + \sum_{k_1+ \dots + k_m=n}\mathbb{P}(E_{k_1, ...k_m})\sum_{i=1}^{m}H(\mu_i^{(k_i)}).$$
    Here $(k_1, ..., k_m)\sim \mathrm{Multinomial}(n, \beta_1, \dots, \beta_m)$. We bound it with the entropy of the uniform distribution over the same ground set, which has size equal to the number of $n$-length combinations of $m$ items with repetition. That is $\binom{n+m-1}{m-1}\leq(n+m)^m$, yielding
    $$H(\mu^{(n)}) = \omega_{n, m} + \sum_{k_1+ \dots + k_m=n}\mathbb{P}(E_{k_1, ...k_m})\sum_{i=1}^{m}H(\mu_i^{(k_i)}),$$
    with $0\leq \omega_{n, m} \leq m \log(n+m)$.
    Interchanging the sums concludes the proof.
\end{proof}

\section{Absolutely continuous permutons} \label{sec:abs_cont}

\begin{lemma} \label{lemma:entropy_of_pos_set}
    Let $A$ be a set of positive measure and $\lambda^A = \frac{\lambda|_A}{\lambda(A)}$ be the uniform probability distribution on it. Then $$\lim\frac{H_n(\lambda^A)}{n\log n}=1.$$
\end{lemma}

\begin{proof}
    Assume first that $A$ is bounded, then without loss of generality we can assume that $A\subseteq [0, 1]^2$. Apply the upper bound of Lemma \ref{lemma:entropy_of_convex_comb} for $\lambda^A$ and $\lambda^{[0, 1]^2\setminus A}$ with coefficients $\lambda(A)$ and $1-\lambda(A)$. Then the combined measure is just $\lambda$, thus by the upper bound on $\lowersupent$ we find
    $$\supent(\lambda) = 1\leq \lambda(A)\lowersupent(\lambda^A) + (1-\lambda(A))\uppersupent(\lambda^{[0, 1]^2\setminus A})\leq 1,$$
    which immediately yields that the lower superlinear entropy (and hence the superlinear entropy) of $\lambda^A$ equals 1.

    In the general case, take bounded $A_0\subseteq A$. Then $\lambda^A$ is the convex combination of $\lambda^{A_0}$ and $\lambda^{A \setminus A_0}$ with coefficients $\lambda(A_0)/\lambda(A)$ and $1-\lambda(A_0)/\lambda(A)$. Thus by applying the lower bound of Lemma \ref{lemma:entropy_of_convex_comb} to the first constituting measure, we find
    $$\lowersupent(\lambda^A)\geq \frac{\lambda(A_0)}{\lambda(A)}\supent(\lambda^{A_0}).$$
    Taking $A_0\to A$ and using that we have already proven the statement for bounded sets, the proof is done.
\end{proof}

\begin{proof}[Proof of Theorem \ref{thm:abs_cont}]

    We note that Theorem \ref{thm:abs_cont} actually extends to pre-permutons.

    Let $\mu_{\mathrm{ac}}$ be the non-vanishing absolutely continuous part of $\mu$, and let $g=\frac{d\mu_{\mathrm{ac}}}{d\lambda}$. We can fix $A$ such that $1/K<g<K$ holds restricted to $A$ and $\mu_{\mathrm{ac}}(A), \lambda(A)>0$, while $\mu_{\bot}(A)=0$.

    Note that $\mu$ is trivially the convex combination of the normalized versions of $\mu_{\mathrm{ac}}^A$ and $\mu-\mu_{\mathrm{ac}}^A$ with coefficients $\mu(A)$ and $1-\mu(A)$, and Lemma \ref{lemma:entropy_of_convex_comb} can be applied. Thus we get the lower bound
    $$\lowersupent(\mu)\geq \mu(A)\lowersupent(\mu_{\mathrm{ac}}^A)$$

    By Lemma \ref{lemma:abs_cont_bound} and the bound on the Radon-Nikodym derivative $g$ on $A$, we find
    that actually
    $$\lowersupent(\mu)\geq \mu(A)\lowersupent(\lambda^A).$$
    As $\mu(A)$ can be arbitrarily close to $\|\mu_{\mathrm{ac}}\|$, this concludes the proof due to Lemma \ref{lemma:entropy_of_pos_set}.
    
\end{proof}

\section{Measure-preserving functions with regularity properties} \label{sec:kolmogorov_sinai}

Our goal is to prove Theorem \ref{thm:differentiable_maps}. We increase generality through a sequence of lemmas for better readability.

\begin{lemma} \label{lemma:diff_maps_piecewise_lin}
    The statement of Theorem \ref{thm:differentiable_maps} holds for $f$ with finitely many linear pieces.
\end{lemma}

\begin{proof}
    Partition the range $[0,1]$ into non-overlapping intervals $J_1, ..., J_l$ such that each $f^{-1}(J_j)$ consists of finitely many intervals on which $f$ is linear and mapped to $J_j$ bijectively by $f$. For fixed $j$, denote these intervals by $I_j^{1}, ..., I_j^{k_j}$. The intervals $((I_j^i)_{i=1}^{k_j})_{j=1}^l$ induce a partition of the domain $[0, 1]$, and by the \textit{pieces of $f$} we will mean $f$ over intervals of this partition. The number of pieces is $k=\sum_{j=1}^{l}k_j$, and each $k_j\leq \kappa$, being the number of monotonicity intervals of $f$. That is, $\kappa$ is independent of the partition $J_1, \dots, J_l$.

    Let $\mu=\mu_f$. First let us understand the pattern distribution of $n$-samples conditioned on that a given number $m_j^i$ of points fall on the piece of $f$ over $I_j^i$. Denote the vector of these numbers by $\mathbf{m}$ and the corresponding distribution by $\mu^{\mathbf{m}}$. The sample points falling into $f^{-1}(J_j)\times J_j$ determine a random pattern of length $\sum_{i=1}^{k_j} m_j^i$, denote its distribution by $\mu^{\mathbf{m}}_j$. Then it is easy to see that $\mu^{\mathbf{m}}$ is just the independent coupling of the $\mu^{\mathbf{m}}_j$s, yielding $H(\mu^{\mathbf{m}}) = \sum_{j=1}^{l}H(\mu^{\mathbf{m}}_j)$. The random pattern $\pi^{\mathbf{m}}_j\sim \mu^{\mathbf{m}}_j$ consists of $k_j$ monotone parts, each consisting of $m_j^i$ entries, the only source of randomness being the interlacement of these parts. Hence its distribution for some fixed $j$ is identical to the solution of the following classical problem: we drop a given number of points from uniform distribution, each of them deterministically labeled with $1, ..., k_j$, and then we proceed from 0 to 1 and note the sequence of labels we encounter, what is the distribution of the label sequence we noted? The answer is that we see the uniform distribution over the admissible sequences, as it is immediate for pairwise distinct labels, and identifying certain labels glues together a uniform number of outcomes. The number of admissible sequences in this particular case is $M_j^{\mathbf{m}} = \frac{(\sum_{i=1}^{k_j}m_j^i)!}{\prod_{i=1}^{k_j}m_j^i!}$, hence $\mu^{\mathbf{m}}_j$ is uniform over a ground set of this size. Taking the independent coupling for $j=1, ..., l$, we get that the pattern distribution conditioned on that $m_j^i$ points fall on the piece of $f$ over $I_j^i$ is uniform over $\prod_{j=1}^l M_j^{\mathbf{m}}$ different patterns, thus $H(\mu^{\mathbf{m}})=\sum_{j=1}^l\log M_j^{\mathbf{m}}$.

    Note that the true pattern distribution of the $n$-samples can be expressed as
    $$\mu^{(n)} = \sum_{\mathbf{m}}\mathbb{P}(\mathbf{m})\mu^{\mathbf{m}}.$$
    Thus by Lemma \ref{lemma:basic_entropy_estimates}, we find
    \begin{equation}\label{eq:piecewise_lin_convex_comb_ineq_invoked}
    \sum_{\mathbf{m}}\mathbb{P}(\mathbf{m})H(\mu^{\mathbf{m}})\leq H(\mu^{(n)})\leq H(\mathbf{m}) + \sum_{\mathbf{m}}\mathbb{P}(\mathbf{m})H(\mu^{\mathbf{m}})
    \end{equation}
    Note that the number of admissible sequences $\mathbf{m}$ is bounded by $n^k$, thus $H(\mathbf{m})\leq k\log n=o(n)$. Hence
    $$\linent(\mu)= \lim_n \frac{\sum_{\mathbf{m}}\mathbb{P}(\mathbf{m})H(\mu^{\mathbf{m}})}{n},$$
    if these limits exist, thus it would suffice to understand the asymptotics of 
    \begin{equation} \label{eq:decomp_piecewise_lin_ent}
    \sum_{\mathbf{m}}\mathbb{P}(\mathbf{m})H(\mu^{\mathbf{m}})=\sum_{\mathbf{m}}\mathbb{P}(\mathbf{m})\sum_{j=1}^l\log M_j^{\mathbf{m}}
    \end{equation}

    This can be vastly simplified as $\mathbf{m}$ is highly concentrated:

    \begin{claim}\label{claim:multinomial_concentration}
        If $\varepsilon>0$, then for $n$ large enough,
        $$\mathbb{P}(\forall i,j: m_j^i \in (n (|I_j^i|-\varepsilon/n^{1/3}), n (|I_j^i|+\varepsilon/n^{1/3}))>1-\varepsilon.$$
    \end{claim}

    \begin{proof}[Proof of Claim \ref{claim:multinomial_concentration}]
        It is sufficient to verify that
        \begin{equation} \label{eq:multinomial_conc}
        \sum_{i, j} \mathbb{P}(m_j^i \notin (n (|I_j^i|-\varepsilon/n^{1/3}), n (|I_j^i|+\varepsilon/n^{1/3}))<\varepsilon.
        \end{equation}
        Observe that $m_j^i\sim \mathrm{Binom}(n, |I_j^i|)$, thus for fixed $i, j$ this probability can be bounded by using Chernoff's bound:
        $$\mathbb{P}(m_j^i \notin (n (|I_j^i|-\varepsilon/n^{1/3}), n (|I_j^i|+\varepsilon/n^{1/3}))< 2 \exp\left(-\frac{-\varepsilon^2n^{1/3}}{3|I_j^i|}\right)<\varepsilon/k,$$
        where the second inequality holds for large enough $n$. As the sum in \eqref{eq:multinomial_conc} consists of $k$ terms, this proves the claim.
    \end{proof}

    Note that conditioned on the event proven to have probability exceeding $1-\varepsilon$ in Claim \ref{claim:multinomial_concentration}, 
    $$\sum_{i=1}^{k_j}m_j^i \in \left(n \sum_{i=1}^{k_j}(|I_j^i|-\varepsilon/n^{1/3}), n \sum_{i=1}^{k_j}(|I_j^i|+\varepsilon/n^{1/3})\right)=(n(|J_j|-k_j\varepsilon/n^{1/3}), n(|J_j|+k_j\varepsilon/n^{1/3})).$$
    
    For such $\mathbf{m}$, by Stirling's approximation we have
    $$\log M_j^{\mathbf{m}}=-n\left(-|J_j|\log |J_j| + \sum_{i}|I_j^i|\log |I_j^i|\right)+o(n),$$
    On the other hand, for any $\mathbf{m}$, recalling that $\kappa$ is the number of the monotonicity intervals of $f$, bounding each $k_j$ from above,
    $$\log M_j^{\mathbf{m}}\leq \sum_{i=1}^{k_j}m_j^i \log \kappa.$$
    Plugging these findings into \eqref{eq:decomp_piecewise_lin_ent}, we find
    \begin{equation} \label{eq:decomp_piecewise_lin_ent2}
    (1-\varepsilon)(-n)\sum_{j=1}^{l}\left(-|J_j|\log |J_j| + \sum_{i=1}^{k_j}|I_j^i|\log |I_j^i|\right)+\varepsilon O(n)= \sum_{\mathbf{m}}\mathbb{P}(\mathbf{m})\sum_{j=1}^l\log M_j^{\mathbf{m}}.
    \end{equation}
    As this holds for arbitrary $\varepsilon$ with large enough $n$, we find
    $$\lim_n \frac{\sum_{\mathbf{m}}\mathbb{P}(\mathbf{m})H(\mu^{\mathbf{m}})}{n} = \sum_{j=1}^{l}\left(|J_j|\log |J_j| - \sum_{i=1}^{k_j}|I_j^i|\log |I_j^i|\right)$$
    Finally, note that we can write $I_j^i= \frac{|J_j|}{\left|f'|_{I_j^i}\right|}$ to find that for fixed $j$, the summand is
    $$\sum_{i=1}^{k_j}|I_j^i| \log \left|f'|_{I_j^i}\right|=\int_{f^{-1}(J_j)}\log |f'|,$$
    which concludes the proof.
\end{proof}

\begin{lemma} \label{lemma:diff_maps_bounded_derivative}
    The statement of Theorem \ref{thm:differentiable_maps} holds for $f$ with a bounded derivative.
\end{lemma}

\begin{proof}
    We can replicate mostly the proof of Lemma \ref{lemma:diff_maps_piecewise_lin}, with some additional error terms.
    
    For any $\varepsilon>0$ we can partition $[0,1]$ into non-overlapping $J_1, ..., J_{l(\varepsilon)}$ as in the proof of Lemma \ref{lemma:diff_maps_piecewise_lin} with the extra condition that for any $I_j^i$ we have that
    $$\frac{\max_{I_j^i} |f'|}{\min_{I_j^i} |f'|}< 1+\varepsilon.$$

    For any $\mathbf{m}$, the measure $\mu^{\mathbf{m}}$ is still the independent coupling of the $\mu_j^{\mathbf{m}}$s which are not uniform anymore over a ground set of $M_j^{\mathbf{m}}$ elements. We claim that it is almost uniform. Indeed, if $y_1^i, \dots, y_{m_j^i}^i$ are the $y$-coordinates of points sampled from $I_j^i$, then these already identify the pattern $\sigma$ determined by these points. Thus we can consider 
    $$T_\sigma = \{(y_s^i) \text{ s.t. they determine } \sigma\}\subseteq \mathbb{R}^{\sum_{i=1}^{k_j}m_i^j}.$$
    For fixed $i$, $y_s^i$ is sampled from a measure $\nu_i$ with density $g(y)=\frac{1}{|f'(f^{-1}(y))| |I_j^i|}$ over $J_j$. By our condition on the partition $J_1, \dots, J_{l(\varepsilon)}$, $\frac{1}{1+\varepsilon}\leq g\leq 1+\varepsilon$
    Thus we find for $\pi_j^{\mathbf{m}}\sim \mu_j^{\mathbf{m}}$
    $$\mathbb{P}(\pi_j^{\mathbf{m}} = \sigma)= \int_{T_\sigma}d\nu_1^{m_1^j}\dots d\nu_l^{m_l^j},$$
    and by change of variables,
    $$(1+\varepsilon)^{-\sum_{i=1}^{k_j}{m_i^j}} \int_{T_\sigma}d\lambda^{\sum_{i=1}^{k_j}{m_i^j}}\leq \int_{T_\sigma}d\nu_1^{m_1^j}\dots d\nu_l^{m_l^j} \leq (1+\varepsilon)^{\sum_{i=1}^{k_j}{m_i^j}} \int_{T_\sigma}d\lambda^{\sum_{i=1}^{k_j}{m_i^j}}.$$
    
    However, the Lebesgue measure of $T_\sigma$ comes from the piecewise affine case when the density $g_i(y)$ is constant, that is $\int_{T_\sigma}d\lambda^{\sum_{i=1}^{k_j}{m_i^j}} = 1/|\sigma|!$. Hence $\mu_j^{\mathbf{m}}$ is almost uniform indeed, putting mass between $(1+\varepsilon)^{-\sum_{i=1}^{k_j}{m_i^j}}/|\sigma|!$ and $(1+\varepsilon)^{\sum_{i=1}^{k_j}{m_i^j}}/|\sigma|!$ to each of the patterns. Thus
    $$H(\mu_j^\mathbf{m})\geq \log M_j^{\mathbf{m}} - \sum_{i=1}^{k_j}{m_i^j} \log(1+\varepsilon),$$
    and hence
    $$H(\mu^{\mathbf{m}})\geq \sum_{j=1}^{l}\log M_j^{\mathbf{m}} - n\log(1+\varepsilon) \geq \sum_{j=1}^{l}\log M_j^{\mathbf{m}}-n\varepsilon.$$
    Thus \eqref{eq:decomp_piecewise_lin_ent2} holds with an extra error term of $n\varepsilon$.

    Furthermore, by the mean value theorem we can find $\xi_{j^i}\in I_j^i$ such that $|I_j^i| = \frac{|J_j|}{|f'(\xi_j^i)|}$, hence
    $$\sum_{j=1}^{l}\left(-|J_j|\log |J_j| + \sum_{i=1}^{k_j}|I_j^i|\log |I_j^i|\right) = -\sum_{j=1}^{l}\sum_{i=1}^{k_j}|I_j^i|\log |f'(\xi_j^i)|.$$
    Thus we can continue \eqref{eq:decomp_piecewise_lin_ent2} as
    \begin{equation} \label{eq:decomp_piecewise_lin_ent3}
    (1-\varepsilon)n\sum_{j=1}^{l}\sum_{i=1}^{k_j}|I_j^i|\log |f'(\xi_j^i)|+\varepsilon O(n)= \sum_{\mathbf{m}}\mathbb{P}(\mathbf{m})\sum_{j=1}^l\log M_j^{\mathbf{m}}.
    \end{equation}
    Dividing by $n$ and taking $n\to\infty$, we see that both the lower and upper linear entropy of $\mu_f$ equals
    $$(1-\varepsilon)\sum_{j=1}^{l}\sum_{i=1}^{k_j}|I_j^i|\log |f'(\xi_j^i)|+\varepsilon O(1).$$
    However, as $\varepsilon\to 0$, this tends to $\int \log |f'|$, concluding the proof.
\end{proof}

\begin{proof}[Proof of Theorem \ref{thm:differentiable_maps}]

    We verify two inequalities by constructing sequences of functions for which the above lemmas can be applied and which approximate the entropy of $f$ from below and above.

    \begin{claim}
        $\upperlinent(\mu_f)\leq \int \log |f'|$.
    \end{claim}
    
    \begin{proof}
    Partition $[0,1]$ into non-overlapping $J_1, \dots, J_{l}$ as in the proof of Lemma \ref{lemma:diff_maps_piecewise_lin}. We will define the piecewise affine approximation $f_N$ of $f$ as follows. First, we equipartition each of the intervals $J_j$ into $N$ subintervals, by abusing the notation we denote the resulting partition by $J_1, \dots, J_l$. Then again following the notation of Lemma \ref{lemma:diff_maps_piecewise_lin}, in each $I_j^i$ we define $f_N$ to be affine and have the same values in the endpoints as $f$.
    Notice that 
    with the notation of the proof of Lemma \ref{lemma:diff_maps_piecewise_lin},
    \eqref{eq:piecewise_lin_convex_comb_ineq_invoked} can be written both for $g=f, f_N$ to see after taking limsup
    $$\upperlinent(\mu_g)= \limsup_n \frac{\sum_{\mathbf{m}}\mathbb{P}(\mathbf{m})H(\mu_g^{\mathbf{m}})}{n},$$
    where $\mathbb{P}(\mathbf{m})$ is independent of the choice of $g$. Notice that for $g=f_N$, the entropy $H(\mu_g^{\mathbf{m}})=\sum_{j=1}^{l} \log M_j^{\mathbf{m}}$, as in the proof of Lemma \ref{lemma:diff_maps_piecewise_lin}, while for $g=f$ it is necessarily smaller. Given this observation, and applying Lemma \ref{lemma:diff_maps_piecewise_lin}, we get
    $$\upperlinent(\mu_f)\leq \upperlinent(\mu_{f_N})=\int \log |f_N'|$$
    for any $N$, with the right hand side tending to $\int \log |f'|$ as $N\to\infty$, due to $f$ being continuously differentiable, proving the claim.
    \end{proof}

    \begin{claim}
        $\lowerlinent(\mu_f)\geq \int \log |f'|$.
    \end{claim}

    \begin{proof}
    Note that $|f'|$ can only diverge to infinity at the boundary of the continuously differentiable pieces. Thus for any $\varepsilon>0$, we can  partition $[0,1]$ into non-overlapping $J_1, \dots, J_{l(\varepsilon)}$ as in the proof of Lemma \ref{lemma:diff_maps_piecewise_lin} with the extra property that there exists $\mathcal{J}\subseteq [l(\varepsilon)]$ with
    $$\sum_{j\in \mathcal{J}}|J_j|<\varepsilon,$$
    and $|f'|$ being bounded in $I_j^i$ for $j\notin \mathcal{J}$. Now define $f_{\varepsilon}$ by only modifying $f$ in the intervals $I_j^i$ for $j\in \mathcal{J}$: in these intervals, let $f_{\varepsilon}$ have derivative 1 and place the pieces so that $f_{\varepsilon}$ is still measure-preserving, and is increasing restricted to $f^{-1}(J_j)$.

    As before, \eqref{eq:piecewise_lin_convex_comb_ineq_invoked} can be written both for $g=f, f_{\varepsilon}$ to see after taking liminf
    $$\lowerlinent(\mu_g)= \liminf_n \frac{\sum_{\mathbf{m}}\mathbb{P}(\mathbf{m})H(\mu_g^{\mathbf{m}})}{n},$$
    where $\mathbb{P}(\mathbf{m})$ is independent of the choice of $g$. Here $H(\mu_g^{\mathbf{m}})=\sum_{j=1}^{l} H((\mu_g)_j^{\mathbf{m}})$. These entropies are independent of the choice of $g$ for $j\notin \mathcal{J}$, as $f=f_{\varepsilon}$ in these regions. On the other hand, for $j\in \mathcal{J}$, $$H((\mu_{f_{\varepsilon}})_j^{\mathbf{m}})=0\leq H((\mu_f)_j^{\mathbf{m}}),$$ as only the identical pattern occurs for such $j$ and $f_{\varepsilon}$. By this observation and by Lemma \ref{lemma:diff_maps_bounded_derivative}, we have
    $$\lowerlinent(\mu_f)\geq \lowerlinent(\mu_{f_\varepsilon})=\int \log |f_{\varepsilon}'|$$
    for any $\varepsilon>0$, with the right hand side tending to $\int \log |f'|$, as $\varepsilon\to 0$, proving the claim.
    \end{proof}
    Combining the two claims concludes the proof.
    
\end{proof}

\section{Entropy of generic functions} \label{sec:generic}

\begin{lemma} \label{lemma:supremum_conv_yields_permuton_conv}
    If $f_n\to f$ in the supremum norm, then $\mu_{f_n}\to \mu_f$.
\end{lemma}

Note that the reverse is not true: if $f$ is the identity and $f_n$ is the skew tent map which takes the value 1 in $1-1/n$, then $f_n\nrightarrow f$, but it is straightforward to check that $\mu_{f_n}\to \mu_f$. Indeed, for fixed $n$ with probability tending to 1 all sample points fall to the first linear piece.

\begin{proof}[Proof of Lemma \ref{lemma:supremum_conv_yields_permuton_conv}]
    We use the third equivalent description given in Proposition \ref{prop:equiv_topologies} to metrize the topology of permutons. Fix an arbitrary rectangle $R=[x_1, x_2]\times [y_1, y_2]$, and for $\varepsilon>0$ let $R_\varepsilon = [x_1, x_2]\times [y_1+\varepsilon, y_2-\varepsilon]$. Then due to having uniform marginals, for any $g\in C(\lambda)$ we have
    $$\mu_g(R_\varepsilon)\leq \mu_g(R)\leq \mu_g(R_\varepsilon)+2\varepsilon.$$
    For large enough $n$, we have $\|f_n - f\|_{\infty}< \varepsilon$, and then
    $$\mu_f(R_\varepsilon)\leq \mu_{f_n}(R),$$
    which together with the previous observation yields
    $$\mu_f(R)-2\varepsilon \leq \mu_{f_n}(R).$$
    The role of $f_n, f$ is symmetric here, thus we have the same inequality with the role of them switched, hence for large enough $n$
    $$d_{\square}(\mu_f, \mu_{f_n})\leq 2\varepsilon.$$
    This yields the desired conclusion.
\end{proof}

\begin{proof}[Proof of Theorem \ref{thm:generic_lack_of_entropy}] 

Let $X_{N,\varepsilon}\subseteq C(\lambda)$ (resp. $Y_{N, \varepsilon}\subseteq C(\lambda)$) be the set of functions for which there exists $n>N$ with $H_n< \varepsilon n $ (resp. $H_n>(1-\varepsilon) n\log n$). Due to Lemma \ref{lemma:supremum_conv_yields_permuton_conv} and the continuity of $H_n$, both $X_{N,\varepsilon}, Y_{N, \varepsilon}$ are open sets for any choice of parameters. Proving their density would conclude the proof.

The density of $Y_{N, \varepsilon}$ is quite simple. For a given partition to affine pieces of a piecewise affine function $f$, we define the $k$-oscillating approximator $f_k$ by replacing each of its affine pieces with $2k+1$ affine pieces of equal length, each having the same range as the original affine piece, and having the same values in the endpoints. For given $\delta>0$, by using a sufficiently fine partition, we can guarantee that any $k$-oscillating approximator is $\delta$-close to $f$. Thus due to the density of piecewise affine functions in $C(\lambda)$ \cite[Proposition~8]{Bobok_2020}, it suffices to prove that for any piecewise affine function $f$, any valid partition and large enough $k$ we have $f_k\in Y_{N, \varepsilon}$. However, by taking $k\to\infty$ it is simple to see that $\mu_{f_k}$ weakly converges to $\lambda^A$ where $A$ is the union of some rectangles. However, for large enough $n$ we have $H_n(\lambda^A)>(1-\varepsilon)n\log n$ by Theorem \ref{thm:abs_cont}. By continuity, we have the same lower bound for $H_n(\mu_{f_k})$ for large enough $k$. It verifies the density of $Y_{N, \varepsilon}$.

Proving the density of $X_{N,\varepsilon}\subseteq C(\lambda)$ is a bit more technical. Consider the dense subspace of continuous piecewise affine measure-preserving maps, $\mathrm{CPA}(\lambda)\subseteq C(\lambda)$. Note that for a piecewise affine function $f$, due to Theorem \ref{thm:differentiable_maps}, $\mu_f$ has superlinear entropy 0, thus we immediately obtain that $\lowersupent \mu_f = 0$ in a dense set of functions, and a weaker form of Theorem \ref{thm:generic_lack_of_entropy} about the genericity of $\lowersupent \mu_f = 0$ readily follows. To prove the same about $\lowerlinent$, we have to construct a dense set of functions with low linear entropy in $\mathrm{CPA}(\lambda)$.

The idea is the following: for $f\in \mathrm{CPA}(\lambda)$, we can find a typically discontinuous piecewise affine measure-preserving function $g\in\mathrm{PA}(\lambda)$, which is very near $f$ and $|g'|=1$ wherever $g$ is continuous. By Theorem \ref{thm:differentiable_maps}, this has linear entropy 0. Let us replace each piece of $g$ by shrinking its domain a bit but leaving its range intact, that is increase the derivative slightly. This creates gaps in the domain of $g$, fill these in a piecewise affine manner so that the resulting function $h\in\mathrm{CPA}(\lambda)$. It is intuitively clear that by going to 0 with the size of the gaps, for fixed $n$ the pattern densities of $\mu_h$ will tend to the ones of $\mu_g$, as sampling $n$ points barely ever hit the graph of the function above a gap, thus $h$ can witness that $X_{N, \varepsilon}$ is dense if we choose the parameters during this argument suitably.

    Let us check that the steps explained above can be carried out. A visualization of the construction can be seen in Figure \ref{fig:low_entropy_construction}. First, we need that $g$ with the prescribed properties can be found in $B(f, \delta)$. To this end, partition the range $[0,1]$ into non-overlapping intervals $J_1, ..., J_l$ such that each $f^{-1}(J_j)$ consists of finitely many intervals in which $f$ is affine, and $|J_j|<\delta$ for each $j$. The role of the indices is symmetric, thus to simplify notation, we will focus on $J=J_1$ and explain how to construct $g$ and $h$ restricted to that. Denote the aforementioned intervals of $f^{-1}(J_1)$ by $I_1, ..., I_k$. In these intervals, we can successively define $g$ to have a single piece of slope $\pm 1$, sign in accordance with the one of $f'$, such that the ranges exhaust $J$. The resulting $g\in \mathrm{PA}(\lambda)\cap B(f, \delta)$. Now in $I_m$ the closed intervals $I_m^{-},I_m^{0}, I_m^{+}$ having the following properties are uniquely defined:
\begin{itemize}
    \item they follow each other in this order, non-overlapping, and their union is $I_m$.
    \item for a constant $\alpha$ to be fixed later, $|I_m^{-}| = \alpha |I_m|\frac{\sum_{i<m}|I_i|}{|J|}$ and $|I_m^{+}| = \alpha |I_m|\frac{\sum_{i>m}|I_i|}{|J|}.$
\end{itemize}
We define $h$ in each of these subintervals to be affine and having the same monotonicity as $g$ and $f$. More precisely, in $I_m^{0}$ we set $h$ such that $h(I_m^{0})=g(I_m)$, while in the other two intervals we set $h$ such that in $h(I_m)=f(I_m)=J$. Note that $h'=\frac{|J|}{\alpha|I_m|}$ in $I_m^{-}$ and $I_m^{+}$ for any $m$. We can also express
$$|I_m^{0}| = |I_m|\left(1-\alpha \frac{\sum_{i\neq m}|I_i|}{|J|}\right)=|I_m|\left(1-\alpha \frac{|J|-|I_m|}{|J|}\right).$$

\begin{figure}[h]
\centering
\includegraphics[width=12cm]{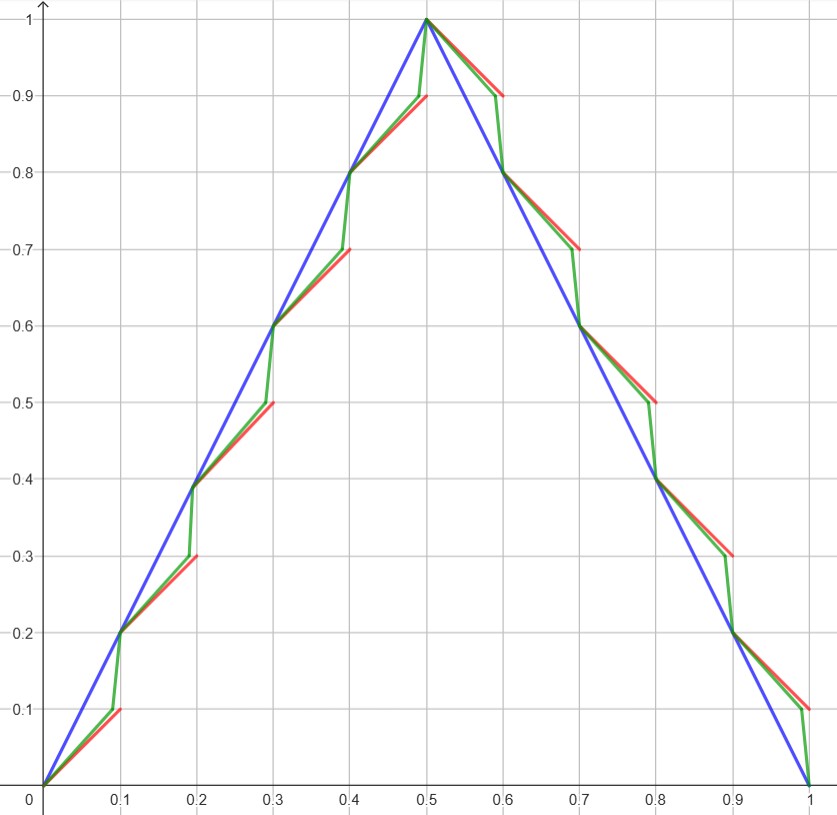}
\caption{Visualizing $g$ (red) and $h$ (green) for $f$ (blue) being the symmetric tent map. The interval $J_i$ ($i=1, \dots, 5)$ is fixed to be $J_i = [0.2\cdot(i-1), 0.2\cdot i]$, while $\alpha=0.9$.  \label{fig:low_entropy_construction}}
\end{figure}

It is immediate that $h$ is continuous as inside any such $I_m$ it is continuous, and in the endpoints it admits the same values as $f$. It is also clear that $h\in B(f, \delta)$ as $f(x)\in J_j$ for some $j$ yields $h(x)\in J_j$. Finally, it is also measure-preserving: first, it is sufficient to check the measure-preserving property for subintervals of $J$. Moreover, actually it is sufficient to check this property for the intervals $g(I_m)$ ($m=1, ..., k$), as the preimage of any such interval consists of finitely many intervals in which $h$ is piecewise affine. Observe that for $m\neq m'$, $h^{-1}(g(I_m))\cap I_{m'}\subseteq I_{m'}^{-}\cup I_{m'}^{+}$ has measure $\alpha \frac{|I_m|||I_{m'}|}{|J|}$, since $h$ acts on $I_{m'}^{-}\cup I_{m'}^{+}$ with a uniform derivative, while $h^{-1}(g(I_m))=I_m^0$. Thus we find
$$h^{-1}(g(I_m)) = |I_m|\left(1-\alpha \frac{|J|-|I_m|}{|J|}\right) + \sum_{m'\neq m}\alpha \frac{|I_m|||I_{m'}|}{|J|}=|I_m|.$$
Thus $h\in\mathrm{CPA}(\lambda)$ indeed. It remains to calculate its entropy, which is $\int_0^{1}\log |h'|$ according to Theorem \ref{thm:differentiable_maps}. We claim that for small enough $\alpha$, it is at most $\varepsilon$, that would conclude the proof. Constants $C_1, C_2, ...$ to be introduced from now on may depend on $\delta, f, (J_j)_{j=1}^{l}$, as these are fixed by when $\alpha$ is introduced.

As the number of intervals $(J_j)_{j=1}^{l}$ and number of base intervals of $f^{-1}(J_j)$ is fixed, it is sufficient to see $\int_{I_m}\log |h'|<C_1\varepsilon$. Partition this integral further to integrals over $I_m^{-}, I_m^{0}, I_m^{+}$. Notice that the length of $I_m^{-}, I_m^{+}$ is at most $\alpha$, and the derivative is at least $\frac{C_2}{\alpha}$. Thus the integral over these intervals is bounded by $C_3\alpha \log \frac{C_2}{\alpha}\to 0$, as $\alpha\to 0$. On the other hand, $h'\to 1$ in $I_m^{0}$ as $\alpha\to 0$, thus the integral over $I_m^{0}$ also goes to 0. We conclude that for small enough $\alpha$ we have $\int_{I_m}\log |h'|<C_1\varepsilon$.
\end{proof}

\section{Permutons from random automorphisms of the $d$-ary tree} \label{sec:random}

Following the introduction and Section \ref{subsection:random_aut_notation}, fix some $F\in \PP(\Sym(d))$, denote by $\Pi$ its support, and by $f$ the induced random measure-preserving bijection. In this section, $f$ is always understood as this random function. We have seen how $f$ can be identified with a random automorphism of the infinite rooted $d$-ary tree. Our ultimate goal in this section is to prove Theorem \ref{thm:random_automorphism}. The proof consists of two core parts: first, by giving an exponentially decaying large deviation bound on $H_n(\mu)/n$ we reduce the problem to the analysis of the means, second, we prove that these form an asymptotically log-periodic sequence with the desired asymptotics. The first of these steps is quite standard, but the second one needs some preparations. As another auxiliary note, we also look a bit into the pattern structure of $\mu_F$. These preparations form the next two subsections. 

\subsection{Patterns of the random automorphism}

Let $X_\Pi$ be the class of permutations which can be obtained from the trivial permutation of one element by finitely many steps of taking one entry and replacing it by an arbitrary subpermutation of some $\pi \in \Pi$. For example, if $d=2$, all length 3 permutations are in $X$, but $3142\notin X$.

\begin{lemma} \label{lemma:automorphism_finite_lin_entropy}
The followings hold for $X=X_\Pi$ and $\mu=\mu_f$:
\begin{enumerate}
    \item $X$ is a substitution-closed permutation class and $\pi_n(f)\in X$.
    \item If $\sigma\notin X$, then $t(\sigma, \mu_f)=0$ for any outcome.
    \item If $\sigma \in X$, then $t(\sigma, \mu_f)>0$ with probability 1.
    \item $X$ is a pattern-avoiding class.
    \item For some constant $K\geq 0$ depending only on $d$, $H_n(\mu_f)\leq Kn$, and hence $\upperlinent(\mu_f)<\infty$.
\end{enumerate}
\end{lemma}

\begin{proof}
    Let $\mu=\mu_f$, $\pi_n(f)=\pi_n$.

    (1) is trivial, and (2) immediately follows from it due to $\pi_n(f)\to \mu_f$.

    For (3), observe that as $\sigma \in X$, for $n$ with $dn\geq |\sigma|$, $\pi_n=\pi_n(f)$ can be built such that $\sigma$ is a pattern of $\pi_n$. In this case, if we sample $|\sigma|$ points from $\mu_f$, with positive probability the sample points are all in the corresponding base squares of $\mu_f$, that is $t(\sigma, \mu_F)>0$. However, upon building $f$ there are infinitely many base squares in which the next $n$ steps of the sampling of $f$ can look like $\pi_n$. Thus at least one of these events occur with probability 1.
    
    For (4), take the pattern (in one-line notation)
    $$\sigma = (2d-1, 2d-3, \dots, 1, 2d, 2d-2, \dots, 2)\in \Sym(2d)$$
    that is we list all the odd, then all the even numbers in decreasing order. We claim that $\sigma$ can appear in $\mu$ only if all the sample points are in the same base square. Indeed, if we sample $2d$ points $p_1, \dots, p_{2d}$ from $\mu$, each base square $Q$ contains a subset of them with consecutive $x$ and $y$ coordinates. This yields that if $\tau$ denotes the pattern determined, the integer interval $[2d]$ can be partitioned into $d$ disjoint, possibly degenerate integer intervals $I_1, ..., I_{d}$ such that the intervals $\tau(I_k)$ are disjoint. It is easy to see that for $\tau=\sigma$, only trivial partitions, with all but one $I_k$ being empty have this property. However, by the same argument if $\sigma$ appear in one of the base squares $\pi_1$, all sample points have to be in the same base square of $\pi_2$. By induction, the same can be told at any level, which yields that $t(\sigma, \mu)=0$. 
    
    (5) follows from (4), as $\mu^{(n)}$ is concentrated on $\sigma$-avoiding permutations, and the number of such permutations can be bounded from above by some $C^n$ according to the Marcus--Tardos theorem \cite{MARCUS2004153}, $K= \log C$ is an appropriate choice.
\end{proof}

Note that in the binary case, if $\Pi=\Sym(2)$, $\mu_f$ evades the patterns $3142$ and $2413$, and $\mu_f$ contains precisely the separable permutations. Observe furthermore that (5) of Lemma \ref{lemma:automorphism_finite_lin_entropy} also quickly follows from the explicit formula provided by Theorem \ref{thm:random_automorphism}, without referring to the Marcus--Tardos theorem.

\subsection{Self-averaging sequences and log-periodicity} \label{sec:self_avg}

The following stochastic process will be relevant: starting with some integer $Y_0$, consider the process $Y_{p+1}\sim \Binom(Y_p, q)$, and let
$$\delta_l(n)=\PP(\exists p: Y_p=l\ |\ Y_0=n).$$
We call the process $Y_p$ {\it simple binomial decay}. While one would probably expect that the sequence $\delta_l(n)$ converges, in \cite{ESS1993} it is proved that it is not the case. Instead, $\delta_l(n)$ is asymptotically log-periodic:

\begin{theorem}[ESS1993, Theorem~2-3., Theorem 2k-3k.]\label{thm:self_averaging_asymptotics}
    The log-periodic limit of $\delta_l(n)$ in base $q^{-1}$ is the smooth function with Fourier expansion 
    $$L_l(x) = \frac{1-q^l}{l! |\log q|}\sum_{r=-\infty}^{\infty}\Gamma\left(l+\frac{2\pi i r}{\log q}\right)\exp (2\pi i r x)$$
\end{theorem}

\begin{megj}
    By inspecting the Fourier coefficients, we can see that $L_l(x)$ is not constant, hence $\delta_l(n)$ is not convergent!
\end{megj}

In \cite{ESS1993} this story is told differently: $n$ players independently score points according to geometric distribution, what is the probability $\widetilde{\delta}_l(n)$ of having $l$ players tied for the first place? Observe that the number of players having at least $p$ points is distributed as $Y_p$, thus
\begin{equation}\label{eq:reparam}
\widetilde{\delta}_n = \PP(\exists p: Y_p=l, Y_{p+1}=0\ |\ Y_0=n) = \frac{(1-q)^l}{1-q^l}\delta_n.
\end{equation}
This equivalence of these contexts is also noted in \cite{ESS1993}, who refer to the problem \cite{Rade} for the binomial decay context, solved in this form about the same time by \cite{Mann}. The formula we present in Theorem \ref{thm:self_averaging_asymptotics} takes into account the reparametrization \eqref{eq:reparam}. We also note \cite{ESS1993} formulates this limiting behaviour along subsequences, essentially not noting the uniformity of convergence our definition of asymptotic log-periodicity requires, but it is easy to see that their proof implies this as well.

We will consider series of log-periodic functions. Note that in general infinite sums of asymptotically log-periodic sequences might fail to be asymptotically log-periodic due to different convergence speeds. We give the following sufficient condition:

\begin{lemma}\label{lemma:sum_log_periodic}
    Assume that $a_i(n)$, $i=1, 2, \dots$ are asymptotically log-periodic sequences with base $\eta$ and log-periodic limits $A_i(x)$. If the series $\sum_{i=1}^{\infty}a_i(n)$ is uniformly convergent, then the series $\sum_{i=1}^{\infty}A_i(x)$ is also uniformly convergent, and $a(n)=\sum_{i=1}^{\infty}a_i(n)$ is also asymptotically log-periodic with log-periodic limit $A(x)=\sum_{i=1}^{\infty}A_i(x)$.
\end{lemma}

\begin{proof} 
    Fix $\varepsilon>0$. By the assumption on the uniform convergence, we can find $i_0$ such that for any $i_1, i_2>i_0$ and every $n$
    $$\left|\sum_{i=i_1}^{i_2}a_i(n)\right|<\varepsilon.$$
    Consequently,
    $$\left|\sum_{i=i_1}^{i_2}A_i(x)\right| = \lim_{m\to\infty} \left|\sum_{i=i_1}^{i_2}a_i([\eta^{x+m}])\right|<\varepsilon$$
    for every $x$. This shows that the series $\sum_{i=1}^{\infty}A_i(x)$ is uniformly Cauchy, $A(x)=\sum_{i=1}^{\infty}A_i(x)$ is a locally uniform limit.

    To identify $A(x)$ with the log-periodic limit of $a(n)$, note that for fix $\varepsilon>0$ we can find $M_0$ such that for any $M>M_0$ and $n>0, x\in [0, 1]$
    $$\left|\sum_{i=M}^{\infty}a_i(n)\right|<\varepsilon,\ \left|\sum_{i=M}^{\infty}A_i(x)\right|<\varepsilon.$$
    Moreover, by the definition of log-periodic limits, we can find $m_0$ such that for any $m>m_0$ and $x\in [0, 1]$
    $$\left|\sum_{i=1}^{M}a_i([\eta^{x+m}]) - \sum_{i=1}^{M}A_i(x)\right|<\varepsilon.$$
    Combining these observations, we find for any $m>m_0$ and $x\in[0, 1]$ that
    $$\left|\sum_{i=1}^{\infty}a_i([\eta^{x+m}]) - \sum_{i=1}^{\infty}A_i(x)\right|<3\varepsilon,$$
    yielding that $A(x)$ is indeed the log-periodic limit of $a(n)$.
\end{proof}

\subsection{Proof of Theorem \ref{thm:random_automorphism}}

\begin{proof}[Proof of Theorem \ref{thm:random_automorphism}]
    Denote by $X_n$ the random variable $H_n(\mu)$. Due to (5) of Lemma \ref{lemma:automorphism_finite_lin_entropy}, we know that $0\leq X_n \leq K$.
    
    Most of the proof is covered by the next two claims:
    \begin{claim}\label{claim:conv_of_means}
        The sequence $\EE X_n /n$ is log-periodic in base $d=1/q$, with the log-periodic limit being smooth and having Fourier expansion
        $$\sum_{r=-\infty}^{\infty}\left(\sum_{l=1}^{\infty}\rho_{l+1}\frac{1-q^l}{(l+1)! |\log q|}\Gamma\left(l+\frac{2\pi i r}{\log q}\right)\right)\exp (2\pi i r x),$$
        for some sequence $\rho_l = O(\log l)$.
    \end{claim}
    \begin{claim}\label{claim:large_dev_principle}
        Let $Y_n = X_n/n$. For fixed $s>0$, there exists some $m_0$ such that if $m\geq m_0$ and $n>d^{2m}$, we have
        $$\PP\left(\left| Y_n - \EE Y_n\right|\geq s\right)\leq 2\exp\left(-\frac{s^2 d^m}{2K^2}\right).$$
    \end{claim}
    Proving these claims would be sufficient. Indeed, by Claim \ref{claim:conv_of_means}, it suffices to check that for fixed $t>0$, almost surely only finitely many of the events $\{ \left| Y_n - \EE Y_n\right|\geq s\}$ occur. By Borel--Cantelli lemma, this follows from $\sum_{n=1}^{\infty}p_n<\infty$ for
    $$p_n = \PP\left(\left| Y_n - \EE Y_n\right|\geq s\right).$$
    This is equivalent to $\sum_{n=d^{2m_0}+1}^{\infty}p_n<\infty$. For such $n$, we can fix $m$ with $d^{2m}<n\leq d^{2(m+1)}$, and invoke Claim \ref{claim:large_dev_principle} with this $m$ to find
    $$\sum_{n=d^{2m_0}+1}^{\infty}p_n\leq \sum_{m=m_0}^{\infty}2d^{2m} \exp\left(-\frac{s^2 d^m}{2K^2}\right)<\infty.$$

    We will use the geometrically separated structure of $\mu$ in the proof of both claims. Looking at the base squares of $\pi_m$, let $\nu_{1}, ..., \nu_{d^m}$ denote the normalized restrictons of $\mu$ to these, respectively. These pre-permutons are pairwise geometrically separated. Upon sampling $n$ points from $\mu$, denote the random number of points ending up in the $i$th base square of $\pi_m$ by $k_i$. Then according to Lemma \ref{lemma:geom_sep}, with a random variable $Z_{n,m}=Z_{n,m}(\mu)$ with $$0<Z_{n, m}\leq d^m\log (n+d^m)$$
    we have
    \begin{equation}\label{eq:entropy_decomp}
    X_n = Z_{n, m} + \sum_{i=1}^{d^m} \sum_{k_i=0}^{n}\binom{n}{k_i} d^{-mk_i}(1-d^{-m})^{n-k_i}H_{k_i}(\nu_i)=Z_{n, m} + \sum_{i=1}^{d^m}V_{i, n, m}.
    \end{equation}
    Notice that the random variables $V_{i, n, m}$ are i.i.d., each of them depending on the measure $\nu_i$ only. Moreover, as the random measure $\nu_i$ has the same distribution as $\mu$, we can plug in $H_{k_i}(\nu_i)\leq K k_i$ to find that
    $$0\leq V_{i, n, m} \leq K \sum_{k_i=0}^{n}\binom{n}{k} d^{-mk_i}(1-d^{-m})^{n-k_i} k_i= \frac{nK}{d^m},$$
    as the latter sum is just the expectation of a random variable with distribution $\mathrm{Binom}(n, 1/d^m).$ 

    \begin{proof}[Proof of Claim \ref{claim:large_dev_principle}]
    Due to the choice of $n, m$, we can write actually $$0<Z_{n,m}\leq d^m\log 2n\leq \sqrt n \log 2n.$$
    For $W_{i, n, m}=V_{i, n, m}/n$, we have the bounds
    $$0\leq W_{i, n, m}\leq \frac{K}{d^m},$$
    and
    $$Y_n = \frac{Z_{n,m}}{n} + \sum_{i=1}^{d^m}W_{i, n, m}.$$
    Now by the bounds on $Z_{n,m}$, if $|Y_n - \EE Y_n| \geq s$, then
    $$\left|\sum_{i=1}^{d^m}W_{i, n, m} - \EE\left(\sum_{i=1}^{d^m}W_{i, n, m}\right)\right|\geq s - \frac{\log 2n}{\sqrt n}.$$
    Thus bounding the probability of this event from above bounds the probability in question as well. This is where we specify $m_0$: if $m_0$ is large enough, then for $n>d^{2m_0}$,
    $$s - \frac{\log 2n}{\sqrt n}\geq \frac{s}{2},$$
    making it sufficient to bound
    $$\PP\left(\left|\sum_{i=1}^{d^m}W_{i, n, m} - \EE\left(\sum_{i=1}^{d^m}W_{i, n, m}\right)\right|\geq s/2\right).$$
    We conclude by applying Hoeffding's inequality to the i.i.d. random variables $W_{i, n, m}$, to directly obtain the statement of the claim. 
    \end{proof} 
    
    \begin{proof}[Proof of Claim \ref{claim:conv_of_means}]
    Consider \eqref{eq:entropy_decomp} for $m=1$ and take expectation. Due to $\nu_i$ and $\mu$ being identically distributed, we find
    $$\EE X_n = \EE Z_{n, 1} + d \sum_{k=0}^{n}\binom{n}{k}d^{-k}(1-d^{-1})^{n-k} \EE X_k.$$
    For $n$ large, $0 \leq Z_{n, 1} \leq 2d \log n$. Hence
    writing $y_n=\EE X_n$, and rearranging, this means
    \begin{equation}\label{eq:rearranged}
    y_n = \rho_n + \frac{d^n}{d^{n-1}-1}\sum_{k=0}^{n-1}\binom{n}{k}d^{-k}(1-d^{-1})^{n-k} y_k=\rho_n + \sum_{k=0}^{n-1}\frac{\binom{n}{k}}{d^{n-1}-1}(d-1)^{n-k}y_k,
    \end{equation}
    where $0\leq \rho_n:=\frac{d^{n-1}}{d^{n-1}-1}\EE Z_{n, 1} \leq 2d^2 \log n$. This is the $\rho_n$ being present in the statement, representing the extra entropy coming from representing $\mu^{(n)}$ as a mixture.
    
    Note $y_0=y_1=0$. We are interested in the asymptotics of $y_n/n$. Observe that $y_2=\rho_2$, and unless $F$ is a Dirac measure on $(1, 2, \dots d)$ or $(d, d-1, \dots, 1)$, $H_2(\mu)>0$ almost surely as both 12 and 21 are present in $\mu$ with positive density. This implies that once we prove that there is a log-periodic limit $A$ is of the form proposed by the theorem, then $\int_{0}^{1}A(x)>0$ apart from these degenerate cases, as the coefficient for $r=0$ is positive.

    Notice that specifying $(\rho_n)$ determines $(y_n)$, using the recursion we can explicitly express the contribution of $\rho_l$ to $y_n$ for any $l, n$. Indeed,
    \begin{equation}\label{eq:recursion_y}
    y_n = \sum_{l=2}^{n}\alpha_l(n)\rho_l,
    \end{equation}
    where $\alpha_l(n)=0$ for $n<l$, $\alpha_l(l)=1$, and for $n>l$,
    \begin{equation}\label{eq:recursion_alpha}
        \alpha_l(n)=\sum_{k=0}^{n-1}\frac{\binom{n}{k}}{d^{n-1}-1}(d-1)^{n-k}\alpha_l(k),
    \end{equation}
    that is the coefficient $\alpha_l(n)$ satisfies the same recurrence relation as $y_n$, without the error term. The terms $k=0, \dots, l-1$ vanish, but it will turn out to be useful to write the right hand side this way. The ideal way to think about the coefficient $\alpha_l(n)$ is the following: we independently, uniform at random label the numbers in $[n]$ by $1, \dots, d$, conditioned on not being identically labelled. This labelling corresponds to dropping $n$ points according to  $\mu$ into $[0, 1]^2$ and considering the landing $d^{-1}$-squares. The condition is in accordance with the rearrangement in \eqref{eq:rearranged}. Then we separate label classes and do the same independently to get a second label, and continue this procedure until we get singletons. It is easy to see that $\alpha_l(n)$ is the expected number of $l$-sized label classes appearing in this process, it follows from the recursion on $\alpha$. But
    $$\EE(\text{no. of $l$-sized classes})=\frac{1}{l}\sum_{i=1}^n\PP(\text{$i$ appears in an $l$-sized class})=\frac{n}{l}\PP(\text{1 appears in an $l$-sized class}),$$
    Denote this probability by $\beta_l(n)=\frac{l}{n} \alpha_l(n)$. The size of the class of 1 evolves according to the following process: we start with $X_0 = n$, and given $X_p$, let $X_{p+1}$ have probability mass function $$\PP(X_{p+1}=k)= \frac{\binom{n-1}{k-1}}{d^{n-1}}(d-1)^{n-k}, \\ k=1, \dots, n,$$
    that is $X_{p+1}\sim 1+\mathrm{Binom}(X_p-1, 1/d)$.
    Then
    $$\beta_l(n)=\PP(X_0=n,\exists p\text{ such that }X_p=l).$$
    Shifting the indexing by one, we see the simple binomial decay:
    writing $Y_p = X_p -1$, we simply have $Y_{p+1} \sim \mathrm{Binom}(Y_p, 1/d)$, and with $\delta_l(n) = \gamma_{l+1}(n+1)$, 
    $$\delta_l(n)=\PP(Y_0=n,\exists p\text{ such that }Y_p=l),$$
    which coincides with the definition of $\delta_l(n)$ in Theorem \ref{thm:self_averaging_asymptotics}.

    Rewriting the formula \eqref{eq:recursion_y} in terms of the $\delta$s, and shifting the indexing by one for convenience, we want to understand the asymptotics of
    $$z_{n} = \frac{y_{n+1}}{n+1}=\sum_{l=1}^{n}\delta_l(n)\frac{\rho_{l+1}}{l+1}.$$
    For given $l$, $\delta_l(n)$ is asymptotically log-periodic in base $d$ with limit
    $$\frac{1-q^l}{l! |\log q|}\sum_{r=-\infty}^{\infty}\Gamma\left(l+\frac{2\pi i r}{\log q}\right)\exp (2\pi i r x)$$
    for $q=1/d$. We claim that $z_n$ is asymptotically log-periodic as well due to Lemma \ref{lemma:sum_log_periodic}, to which end the uniform convergence of the series
    $$\sum_{l=1}^{n}\delta_l(n)\frac{\rho_{l+1}}{l+1} = \sum_{l=1}^{\infty}\delta_l(n)\frac{\rho_{l+1}}{l+1}$$
    has to be checked. Divide the range $[1, \infty)$ into intervals of the form $[l, [\sqrt d l]]$, denote these by $I_1, I_2, \dots$. Over such an interval, for some $C>0$
    $$\sum_{l\in I}\delta_l(n)\frac{\rho_{l+1}}{l+1}\leq \sum_{l\in I}\frac{\delta_l(n)}{\min I}\cdot 2D^2 \log(\max I)\leq C\frac{\log(\max I)}{\min I},$$
    yielding a summable series.  (The second inequality follows from the fact that summing $\delta_l(n)$ for $l$ in $I$ corresponds to the expected number of visits of the simple binomial decay in $I$, which is bounded by the choice of $I$.) Thus we find $\sum_{l=l_1}^{l_2}\delta_l(n)\frac{\rho_{l+1}}{l+1}=o(l_1)$, the series $\sum_{l=1}^{\infty}\delta_l(n)\frac{\rho_{l+1}}{l+1}$ is uniformly Cauchy, and $z_n$ is asymptotically log-periodic indeed. The log-periodic limit can be expressed as the sum of the individual log-periodic limits, in which expression the order of the summations can be interchanged due to absolute convergence to obtain the proposed Fourier expansion.  
\end{proof}

    As we have already discussed, Claims \ref{claim:conv_of_means}-\ref{claim:large_dev_principle} conclude the proof of Theorem \ref{thm:random_automorphism}
\end{proof}

The main question now is whether $\EE (H_n(\mu) / n)$ is actually convergent, i.e., whether the log-periodic limit is constant. This happens if and only if all Fourier coefficients for $r\neq 0$ vanish:

\begin{corollary}
    In the setup of Theorem \ref{thm:random_automorphism}, $H_n(\mu)/n$ converges almost surely if and only if 
    $$\sum_{l=1}^{\infty}\rho_{l+1}\frac{1-q^l}{(l+1)! |\log q|}\Gamma\left(l+\frac{2\pi i r}{\log q}\right)$$
    vanishes for all $r\neq 0$.
\end{corollary}

Being aware of this, the hypothetic convergence seem to be miraculous event, determined by the coefficients $\rho_l$: for any fixed $l$ we know that the corresponding term does not vanish, thus for the generic choice of coefficients the series does not sum to 0 either. Of course, in our case $\rho_l$ is a well-defined sequence of coefficients for which this miraculous event might occur.

\section{Perturbed permutons from random automorphisms of the $d$-ary tree} \label{sec:random2}

 Our goal is to prove Theorem \ref{thm:random_automorphism2}. We will deploy the same strategy as in the proof of Theorem \ref{thm:random_automorphism}: we study the asymptotic of the means -- actually proving convergence this time -- and prove concentration. The second of these steps is going to be more challenging compared to the previous scenario due to writing up the relevant random variables as independent sums to invoke concentration inequalities will be possible only after careful conditioning.
 
 We dedicate a subsection to both these steps due to their length. We will use notation introduced in Section \ref{subsection:random_aut_notation}.

Similarly to the previous theorem, the multiple sources of randomness causes some conceptual difficulties. Now we have random node labellings and random edge weights defining $\mu$, and study the entropy defined in terms of randomly sampling from $\mu$. We take note of the fact that the distribution from which $\mu$ is drawn naturally disintegrates based on the choice of the edge weights up to level $m$ to $\tau|_m$. We will denote by $\mu$ the original random permuton, and by $\mu^{t|_m}$ the random permuton arising from the edge weighting of $T_d$ in which weights up to level $m$ are not random but fixed to be $t|_m$. Then for any event $A$ defined in terms of the drawn measure and the location of the sampling points, by Fubini's theorem
\begin{equation} \label{eq:disintegration}
\PP_{\mu}(A)=\EE_{\tau|_m}(\PP_{\mu^{t|_m}}(A)),
\end{equation}
where $\PP_{\mu}(A)$ and $\PP_{\mu^{t|_m}}(A)$ denotes the probability of $A$ upon sampling our random permuton according to the appropriate measure, and $\EE_{\tau|_m}$ stands for $t|_m$ being distributed according to $\tau|_m$.

\subsection{The convergence of the means}

\begin{lemma}\label{lemma:random_aut_conv_means}
    For $\mu=\mu_{F,\Unif}$, $\EE H_n(\mu)/n$ converges.
\end{lemma}

\begin{proof}
    Recall that on the first level, $(\tau_i)_{i=1}^{d}$ are the random side lengths of the squares in the first approximation of $\mu$, being gap sizes in order statistics obtained from $U_1, \dots, U_{d-1}\sim \Unif(0, 1)$. Denote the resulting intervals of the partition by $I_1, \dots, I_d$. Fixing a first-level choice $(\tau_i)_{i=1}^{d} = (t_i)_{i=1}^{d}$, we can write in the spirit of \eqref{eq:entropy_decomp}
    \begin{equation}\label{eq:entropy_decomp2}
    H_n(\mu^{t|_1}) = Z_{n, 1}(\mu^{t|_1}) + \sum_{i=1}^{d} \sum_{k_i=0}^{n}\binom{n}{k_i} t_i^{k_i}(1-t_i)^{n-k_i}H_{k_i}(\nu_i).
    \end{equation}
    Moreover, $0\leq Z_{n, 1}(\mu) \leq 2d\log n$ for $n$ large.
    Note that here
    $$\binom{n}{k_i} t_i^{k_i}(1-t_i)^{n-k_i}=\PP(\text{upon sampling $n$ points from $\mu$, there are $k_i$ points with $x$-coordinate in $I_i$}).$$
    Taking expectation with respect to the random node labelling and nonfixed edge weights, we find due to $\nu_i$ and $\mu$ being equidistributed that
    \begin{equation}\label{eq:entropy_decomp2_expectation}
    \EE[H_n(\mu^{t|_1})]=\EE [Z_{n, 1}(\mu^{t|_1})] + \sum_{i=1}^{d}\sum_{k_i=0}^{n}\PP(\text{there are $k_i$ points in $I_i$})\cdot\EE[H_{k_i}(\mu)]
    \end{equation}
    Taking expectation with respect to the choice of $(t_i)$, due to the gap sizes being equidistributed and having the natural disintegration described in \eqref{eq:disintegration},
    $$\EE [H_n(\mu)] = \EE [Z_{n, 1}(\mu)] + d\sum_{k=0}^{n}\PP(\text{there are $k$ points in $[0, U_{(1)}]$})\cdot\EE [H_k(\mu)].$$
    Note that the distribution on the length $d$ order-dependent partitions of $n$ induced by the partition given by $U_{(0)}, \dots, U_{(d)}$ is uniform, thus we can easily express 
    \begin{equation}\label{eq:point_distribution}
    \PP(\text{there are $k$ points in $[0, U_{(1)}]$})=\frac{\binom{n-k+d-2}{d-2}}{\binom{n+d-1}{d-1}}
    \end{equation}
    by enumerating the corresponding combinations with repetition. 
    
    We check $d=2$ separately, as this case is deceptively simple: \eqref{eq:point_distribution} is just a uniform distribution on $\{0, 1, \dots, n\}$, hence
    \begin{equation}\label{eq:recursion_d=2}
    \EE [H_n(\mu)] = \EE [Z_{n, 1}(\mu)] + \frac{2}{n+1}\sum_{k=0}^{n}\EE [H_k(\mu)],
    \end{equation}
    or after rearrangement and writing $y_n=\EE [H_n(\mu)]$,
    $$y_n = \rho_n + \frac{2}{n-1}\sum_{k=0}^{n}y_k,$$
    where $0\leq \rho_n = \frac{n+1}{n-1}\EE [Z_{n, 1}(\mu)]\leq 4d\log n$. (Note that we have the same upper bound for any $d$ if $n$ is large enough.)
    
    Now in the same fashion as in proceeding from \eqref{eq:rearranged} to \eqref{eq:recursion_alpha} in the proof of Theorem \ref{thm:random_automorphism}, this implies
    $$y_n = \sum_{l=2}^{n}\alpha_l(n)\rho_l,$$
    where $\alpha_l(n)=0$ for $n<l$, $\alpha_l(l)=1$, and for $n>l$,
    $$\alpha_l(n)=\frac{2}{n-1}\sum_{k=0}^{n-1}\alpha_l(k).$$
    Applying this recurrence,
    $$\alpha_l(n+1)= \frac{2}{n}\alpha_l(n)+\frac{2}{n}\sum_{k=0}^{n-1}\alpha_l(k)=\frac{n-1}{n}\alpha_l(n)+\frac{2}{n}\alpha_l(n)=\frac{n+1}{n}\alpha_l(n),$$
    thus $\alpha_l(n)/n$ is constant for $n=l, l+1, \dots$. By plugging in $n=l$, we see that its value is $\frac{1}{l}$, yielding
    $$y_n/n = \sum_{l=2}^{n}\rho_l/l.$$
    It implies that $y_n/n$ is a nondecreasing sequence, whose limit is finite due to 5. in Lemma \ref{lemma:automorphism_finite_lin_entropy}. This concludes the proof for $d=2$.

    The rest of the proof is devoted to $d>2$. By \eqref{eq:point_distribution}, we find that $\alpha_l(n)$ satisfies the following recursion for $n>l$:
    \begin{equation}\label{eq:recursion_alpha_d>2}
    \alpha_l(n)= d\sum_{k=0}^{n}\frac{\binom{n-k+d-2}{d-2}}{\binom{n+d-1}{d-1}}\alpha_l(k),
    \end{equation}
    or after rearrangement,
    \begin{equation}\label{eq:recursion2_alpha_d>2}
    (n+d-1)_{d-1}\alpha_l(n)= d(d-1)\sum_{k=0}^{n}(n-k+d-2)_{d-2}\alpha_l(k),
    \end{equation}    
    where $(x)_m=\prod_{i=0}^{m-1}(x-i)$ is the falling factorial. Recalling
    $$y_n = \sum_{l=2}^{n}\alpha_l(n)\rho_l,$$
    proving the convergence of each $\alpha_l(n)/n$ separately is a promising direction. We note that we can quickly observe that $\alpha_l(n)=n$ (and hence, by linearity, $\alpha_l(n)=Cn$ for any $C$) is a solution. Notably, then the identity in question is
    $$(n+d-1)_{d}= d(d-1)\sum_{k=0}^{n}(n-k+d-2)_{d-2}\cdot k.$$     
    Now the left hand side is the number of ways of choosing an ordered list of length $d$ consisting of distinct elements drawn from a collection of size $n+d-1$. The right hand side counts the same quantity, decomposed based on the location of the second entry: if the second entry is at place $k+1$, there are $(n-k+d-2)_{d-2}\cdot k$ to choose the rest. Thus $\alpha_l(n)=Cn$ is a solution indeed, which is promising for our goal that $\alpha_l(n)$ grows linearly. We need to check essentially that all solutions independent from $Cn$ are sublinear.

    Before dealing with this recurrence relation or difference equation, it is enlightening to study the naturally corresponding differential equation, where the set of solutions can be understood in simpler terms. Replacing the sum by an integral and using the approximation $n+o(n)\approx n$, we can rewrite \eqref{eq:recursion2_alpha_d>2} as
    $$x^{d-1}f(x)=d(d-1)\int_{0}^{x}(x-\xi)^{d-2}f(\xi)d\xi.$$
    Differentiating $d-1$ times with respect to $x$, we obtain a Cauchy--Euler equation of a very simple form:
    \begin{equation}\label{eq:cauchy_euler}
    \left(x^{d-1}f(x)\right)^{(d-1)}=d! \cdot f(x).
    \end{equation}
    This is a linear differential equation of order $d-1$, so there are $d-1$ linearly independent solutions. One should look for solutions of the form $f(x)=x^r$, which is a solution if and only if
    $$\prod_{i=1}^{d-1}(r+i)=d!,$$
    or after writing $s=r+d-1$,
    $$(s)_{d-1}=d!$$
    Finding $d-1$ distinct roots for this polynomial equation gives a basis of the space of solutions of \eqref{eq:cauchy_euler}. If the solutions are not distinct, some extra care is needed, but we do not face this inconvenience. The following claim will have importance in the exact treatment of the recurrence relation as well:

    \begin{claim}\label{claim:falling_factorial_roots}
    The polynomial $p(s)=(s)_{d-1}-d!$ have $d-1$ distinct roots, one of them is $d$, and all other roots have real part smaller than $d$.
    \end{claim}

    \begin{proof}[Proof of Claim \ref{claim:falling_factorial_roots}]
        If $\Re s>d$ or $\Re s = d$ and $s\neq d$, then comparing the absolute values of the terms one-by-one in $(s)_{d-1}$ and $d!=(d)_{d-1}$ shows that $|(s)_{d-1}|>|d!|$. To check that the roots are distinct, observe that if $s\in (-1, d)$, then by the same one-by-one comparison of absolute values $|(s)_{d-1}|<|d!|$. Thus any real root of $p(s)$ is outside of the interval $(-1, d)$. On the other hand, the $d-1$ roots of $q(s)=(s)_{d-1}$ are $s=0, \dots, d-2$, and then by the Gauss--Lucas theorem, all roots of $q'=p'$ are in the interval $[0, d-2]$. Thus there are no common roots of $p, p'$, yielding the existence of $d-1$ distinct roots.
    \end{proof}

    By the previous claim, $\prod_{i=1}^{d-1}(r+i)=d!$ is solved by $r=1$ and all other roots have real part smaller than 1, i.e., any solution of \eqref{eq:cauchy_euler} is of the form $f(x)=Cx+o(1)$. This is the type of statement we want to achieve for the solution of the original recurrence.

    Having made these preliminary observations, we highlight the claim we have already formulated and which we will prove now rigorously:
    
    \begin{claim} \label{claim:alpha_l_n_conv}
    $\alpha_l(n)/n$ is convergent for any $l$. 
    \end{claim}

    \begin{proof}[Proof of Claim \ref{claim:alpha_l_n_conv}]
    Motivated by the way we solved the continuous version, recall the recurrence relation \eqref{eq:recursion_alpha_d>2}
    \begin{equation}\label{eq:recursion2_alpha_d>2_again}
    (n+d-1)_{d-1}\alpha_l(n)= d(d-1)\sum_{k=0}^{n}(n-k+d-2)_{d-2}\alpha_l(k),
    \end{equation}
    and apply discrete differentiation $d-1$ times to get rid of the "integral". Discrete differentiation of a sequence $a(n)$ just accounts for taking $\Delta a(n)=a(n+1)-a(n)$, or after multiple steps, by induction,
    $$\Delta^{(m)}a(n)= \sum_{i=0}^{m}(-1)^{m-i} \binom{m}{i}a(n+i).$$
    Applying this with $m=d-1$ to the right hand side of \eqref{eq:recursion2_alpha_d>2_again}, by definition
    \begin{equation}\label{eq:discrete_diff}
    \Delta^{(d-1)}\left(d(d-1)\sum_{k=0}^{n}(n-k+d-2)_{d-2}\alpha_l(k)\right) =d(d-1)\sum_{i=0}^{d-1}(-1)^{d-1-i} \binom{d-1}{i}\sum_{k=0}^{n+i}(n+i-k+d-2)_{d-2}\alpha_l(k).
    \end{equation}
    Observe that for any $i$, the second sum can be extended to run until $n+d-2$, as plugging in $k=n+i+1, \dots, n+d-2$ vanishes in the falling factorial. Making this extension, decomposing the first sum to $i=d-1$ and $i$ running from $0$ to $d-2$, and exchanging the sums, we see that \eqref{eq:discrete_diff} further equals
    $$d(d-1)\left((d-2)_{d-2}\cdot \alpha_l(n+d-1)+\sum_{k=0}^{n+d-2}\alpha_l(k)\sum_{i=0}^{d-1}(-1)^{d-1-i}\binom{d-1}{i}(n+i-k+d-2)_{d-2}\right).$$
    Observe that in this sum, the coefficient of $\alpha_l(k)$ is just the $(d-1)$th discrete derivative of the degree $d-2$ polynomial $(n-k+d-2)_{d-2}$, up to sign, hence it vanishes. Thus what remains is that the $(d-1)$th discrete derivative of the right hand side of \eqref{eq:recursion2_alpha_d>2_again} is just $d!\cdot\alpha_l(n+d-1)$. Simply plugging the definition of $\Delta$ into the left hand side, we see that taking these derivatives yields the following recurrence of order $d-1$:
    \begin{equation}\label{eq:final_recurrence}
    \sum_{i=0}^{d-1}(-1)^{d-1-i} \binom{d-1}{i}(n+i+d-1)_{d-1}\alpha_l(n+i)=d!\cdot\alpha_l(n+d-1).
    \end{equation}
    Finding $d-1$ linearly independent solutions solves this recurrence. The solutions naturally associated with the polynomial solutions of the corresponding differential equation are the hypergeometric solutions, i.e., we look for solutions of the form $\alpha_l(N+1)/\alpha_l(N)=S(N)$, where $S(N)$ is some rational function. We will be even more restrictive with our choice motivated by what we found in the continuous case, we will look for $S(n)=\frac{N+1}{N+x}$. 
    Then we have 
    $$\alpha_l(n+i)=\alpha_l(n)\cdot \frac{(n+i)_i}{(n+x+i-1)_i}.$$
    Plugging this into \eqref{eq:final_recurrence}, and multiplying by $(n+x+d-2)_{d-1}$, we can simplify by the common factors $(n+d-1)_{d-1}\cdot \alpha_l(n)$ to get a polynomial equation in $n$ and $x$, having degree $d-1$ in both variables:
    \begin{equation}\label{eq:recurrence_x_form}
    \sum_{i=0}^{d-1}(-1)^{d-1-i} \binom{d-1}{i}(n+i+d-1)_{i}\cdot (n+x+d-2)_{d-1-i}=d!.
    \end{equation}    
    Denote the polynomial on the left hand side by $P(x, n)$. The vital observation is that $P(x, n)$ is actually independent of $n$. As we already now that $\alpha_l(n)=n$ solves the recurrence, this $n$-independence holds for $x=0$. Moreover, note that upon plugging in $x=1,2 , \dots, d-1$ into the left hand side, we get
    $$\prod_{j=0}^{x-1}(n+d+x)\sum_{i=0}^{d-1}(-1)^{d-1-i}\binom{d-1}{i}(n+i+d-1)_{d-x}=\prod_{j=0}^{x-1}(n+d+x) \Delta^{d-1}(n+i+d-1)_{d-x}.$$
    This is the product of a degree $x-1$ polynomial and the $(d-1)$th derivative of a degree $d-x$ polynomial, which is zero unless $x=1$, when it is still constant. Thus $P(x, n)$ is independent of $n$ on each of the lines $x=0, 1, \dots, d$. Thus the coefficient of any positive power of $n$ vanishes for $x=0, 1, \dots, d$. But these coefficients are degree $d-1$ polynomials of $x$, hence it is possible only if they are identically zero. Thus $P(x, n)$ is independent of $n$ indeed, and we have already noted that $x=2, \dots, d-1$ are its roots. Recalling \eqref{eq:recurrence_x_form}, we can see that the leading coefficient is $(-1)^{d-1}$, thus 
    \begin{equation}\label{eq:x}
    P(x, n)=(-1)^{d-1}(x-2)_{d-1}=d!,
    \end{equation}
    or plugging in $s=-(x-d)$,
    $$(s)_{d-1}=d!.$$
    We understood this equation in Claim \ref{claim:falling_factorial_roots}: it has $d-1$ distinct roots, including $d$, and all other roots have real part smaller than $d$. As $d-s=x$, this yields that \eqref{eq:x} has $d-1$ distinct roots, including 0, and all other roots have real part larger than 0. Denoting these roots by $x_1, \dots, x_{d-1}$, we have obtained that
    $$\alpha_l^{(i)}(n) = \frac{(n)_{n-l}}{(n+x_i-1)_{n-l}}\alpha_l(l)=\frac{(n)_{n-l}}{(n+x_i-1)_{n-l}}$$
    is a solution for any $i$, $n\geq l$. It is a standard exercise to check that these $d-1$ solutions are linearly independent, thus they span the space of all solutions, and apart from $x_1=1$, all these are sublinear. Thus all solutions are of the form
    $$\alpha_l(n)=\sum_{i=1}^{d-1}c_i\alpha_l^{(i)}(n) =c_1n+o(n).$$
    Thus $\alpha_l(n)/n$ is convergent for any $l$.
    
    These roots yield the $d-1$ linearly independent solutions of \eqref{eq:recursion2_alpha_d>2_again}. 
    \end{proof}

    Now we can conclude the proof of Lemma \ref{lemma:random_aut_conv_means}. Recall that we need that
    \begin{equation}\label{eq:y_n_from_rho}
    y_n/n =\sum_{l=2}^{n}\frac{\alpha_l(n)}{n}\rho_l
    \end{equation}
    is convergent as $n\to\infty$. Due to the convergence of the coefficients $\alpha_l(n)/n$, it suffices to prove that for any $\varepsilon>0$ we can find $l_\varepsilon$ such that for every $n$
    $$\sum_{l=l_\varepsilon}^{n}\frac{\alpha_l(n)}{n}\rho_l<\varepsilon.$$
    To prove this, we observe that $\alpha_l(n)$ has a similar stochastic meaning as in the proof of Theorem \ref{thm:random_automorphism}. In this case, we do the following random process: starting from $n$ stones labelled by $[n]$, we drop $d-1$ separating bars independently uniform at random (they might end up in the same locations and they might fall before or after all the stones). Then in the resulting intervals, we repeat the same process independently until every stone is the sole inhabitant of its interval. We call this the {\it simple partition decay}. Then $\alpha_l(n)$ is the expected number of length $l$ intervals. Here
    \begin{equation}\label{eq:alpha_stochastic_meaning}
    \EE(\text{no. of $l$-long intervals})=\frac{1}{l}\sum_{i=1}^n\PP(\text{$i$ appears in an $l$-long interval}).
    \end{equation}
    Note that in contrast to the other proof, this quantity does not equal
    $\frac{n}{l}\PP(\text{1 appears in an $l$-long interval})$, as the probability $\PP(\text{$i$ appears in an $l$-long interval})$ depends on $i$. (This can be easily seen via the example that the event "1 appears in an $l$-long interval" is contained by the event "2 appears in an $l$-long interval if $l>1$".) Thus we need that for large enough $l_\varepsilon$
    $$\sum_{l=l_\varepsilon}^{n}\frac{1}{nl}\sum_{i=1}^n\PP(\text{$i$ appears in an $l$-long interval})\rho_l<\varepsilon.$$
    We partition the interval $[l_\varepsilon, n]$ into intervals $I_0, I_1, \dots, I_m$, where $I_j =[r^jl_\varepsilon, r^{j+1}l_\varepsilon]$ for some $r>1$. Summing over $l\in I_j$, by the bound on $\rho_l$, we get
    \begin{equation}\label{eq:prob_bound_partition_decay}
    \sum_{l\in I_j}\frac{1}{nl}\sum_{i=1}^n\PP(\text{$i$ appears in an $l$-long interval})\rho_l\leq \frac{4d \log \max I_j}{n \min I_j}\sum_{i=1}^{n}\sum_{l\in I_j}\PP(\text{$i$ appears in an $l$-long interval}).
    \end{equation}
    This sum of probabilities is the expected number of intervals encountered in the simple partition decay containing $i$ and having length in $I_j$, thus we must understand the evolution of this $i$-container length. We can observe that with probability exceeding 1/2, the length of the $i$-container is multiplied by at most $3/4$. (If $d=2$, and $i$ is in the middle of its container, this decay has precisely probability 1/2, and it is plain to see that otherwise it is even larger.) Thus if $r<4/3$, the number of $i$-containers with length in $I_j$ is stochastically dominated by $1+G$ for $G\sim \mathrm{Geo}(1/2)$. Hence the expected number of intervals encountered in the simple partition decay containing $i$ and having length in $I_j$ is at most $\EE(1+G)=3$. Hence \eqref{eq:prob_bound_partition_decay} can be bounded further by
    $$\frac{12d \log \max I_j}{\min I_j}$$
    for any $n$, and hence 
    $$\sum_{l=l_\varepsilon}^{n}\frac{1}{nl}\sum_{i=1}^n\PP(\text{$i$ appears in an $l$-long interval})\rho_l\leq \sum_{j=0}^{m}\frac{12d \log \max I_j}{\min I_j}=\frac{12d}{l_\varepsilon}\sum_{j=0}^{m}\frac{j+1}{r^j}.$$
    The convergence of this series assures that this quantity is indeed at most $\varepsilon$ for $l_{\varepsilon}$ large enough. This concludes the proof.
    \end{proof}

We study the sign of the limit in a separate lemma.

\begin{lemma}\label{lemma:positive_limit}
    For the limit $c=\lim_{n\to\infty}\EE H_n(\mu)/n\geq 0$ guaranteed to exist by Lemma \ref{lemma:random_aut_conv_means}, $c=0$ if and only if $F$ is a Dirac measure on $(1, 2, \dots, d)$ or $(d, d-1, \dots, 1)$.
\end{lemma}

\begin{proof}
    It is obvious that if $F$ is degenerate in the above sense, then the limit vanishes, as then $\mu_{F, \Unif}=\mu_{x\mapsto x}$ or $\mu_{F, \Unif}=\mu_{x\mapsto 1-x}$.
    
    Otherwise recalling \eqref{eq:y_n_from_rho}:
    $$\EE [H_n(\mu)]/n =\sum_{l=2}^{n}\frac{\alpha_l(n)}{n}\rho_l,$$
    it suffices to prove that $\rho_2>0$ and $\lim_{n\to\infty}\frac{\alpha_2(n)}{n}>0$. The first of these claims holds trivially, $\mu_{F, \Unif}$ contains both 12 and 21 with positive density almost surely. We will prove the second claim for any $l$, to this end recall \eqref{eq:recursion_alpha_d>2}:
    \begin{displaymath}
    \alpha_l(n)= d\sum_{k=0}^{n}\frac{\binom{n-k+d-2}{d-2}}{\binom{n+d-1}{d-1}}\alpha_l(k).
    \end{displaymath}
    Rearranging, we get
    $$\alpha_l(n)= d\frac{\binom{n+d-1}{d-1}}{\binom{n+d-1}{d-1}-d}\sum_{k=0}^{n-1}\frac{\binom{n-k+d-2}{d-2}}{\binom{n+d-1}{d-1}}\alpha_l(k),$$
    and for $\beta_l(m)=\alpha_l(m)/m$ (extended as zero to $m=0$),
    \begin{equation}\label{eq:beta}
    \beta_l(n)= \frac{d}{n}\frac{\binom{n+d-1}{d-1}}{\binom{n+d-1}{d-1}-d}\sum_{k=0}^{n-1}k\frac{\binom{n-k+d-2}{d-2}}{\binom{n+d-1}{d-1}}\beta_l(k).
    \end{equation}
    Using the combinatorial identity 
    $$\sum_{k=0}^{n}k \binom{n-k+d-2}{d-2}=\binom{n+d-1}{d},$$
    quick calculation shows that the right hand side of \eqref{eq:beta} is a convex combination of the terms  $\beta_l(k), \ k=0, 1, \dots, n-1$ for $n>l$. This yields that all $\beta_l(n)$ for $n>l$ are positive. (Recall that $\beta_l(n)=0$ for $n<l$, and $\beta_l(l)=1$.) Moreover, as the weights starting from $k=0$ are proportional to $k\cdot P(n-k)$ for a fixed polynomial $P$ of degree $d-2$, it is clear that we can find some $k_0$ fixed such that for $n$ large enough, the total weight put on $\beta_l(0), \dots, \beta_l(l-1)$ is at most the total weight put on $\beta_l(l), \dots, \beta_l(k_0)$. But this implies that 
    $$\beta_l(n)\geq \min \left\{\{\beta_l(k)/2: i\in[l, k_0]\}\cup \{\beta_l(i): i\in[k_0, n-1]\}\right\}>0,$$
    which implies a universal, positive lower bound on $\beta_l(n)$ for $n>l-1$. This verifies $c>0$.
\end{proof}

\subsection{Concentration around the mean}

The following lemma coupled with Lemma \ref{lemma:random_aut_conv_means} about the convergence of the means will immediately yield Theorem \ref{thm:random_automorphism2}:

\begin{lemma} \label{lemma:convergent_probability_sum}
    For every $s>0$, at most finitely many of the events $\{|H_n(\mu) - \EE [H_n(\mu)]| >2ns\}$ occurs.
\end{lemma}

\begin{proof}

 We would like to proceed similarly to how we did in the proof of Theorem \ref{thm:random_automorphism}, however, this time $H_n(\mu)$ is not the sum of $d^m$ independent random variables up to a small error, as the edge weights are not independent. The idea is the following: restrict the edge weights up to level $m$ to make this sum independent. This enables applying Hoeffding's inequality in the same manner, up to some inconveniences. First of all, it will imply high concentration around $\EE [H_n(\mu^{t|_m})]$ instead of $\EE [H_n(\mu)]$. Given that heuristically, $\EE [H_n(\mu)]$ is mostly determined by the later choices, these should mean the same with high probability, but this observation must be made quantitative, i.e., we have to bound the probability that $t|_{m}$ is bad from this aspect. A simpler issue  we have to be aware of is that in the Hoeffding bound, the squared sum $L_m$ of the random sidelengths $l_1, \dots, l_{d^m}$ on the $m$h level will appear, which controls the strength of the concentration. If $L_m$ is too large, our bound might be meaningless, thus we have to bound the probability that $t|_{m}$ is bad from this aspect as well.

\begin{claim}\label{claim:squared_sum_small}
    $\EE L_m = \left(\frac{2}{d+1}\right)^m$, and hence
    $$\PP\left(L_m\geq \left(\frac{2}{d+1/2}\right)^m\right)\leq \left(\frac{d}{d+1/2}\right)^m.$$
\end{claim}

\begin{proof}[Proof of Claim \ref{claim:squared_sum_small}]
    Observe that $L_m =\sum_{i=1}^{d}\tau_i^2L_{m-1}^{(i)}$, where the $\tau_i$ are the sidelengths on the first level and $L_{m-1}^{(i)}$ are independent, identically distributed copies of $L_{m-1}$. Consequently, as $L_1= \sum_{i=1}^{d}\tau_i^2$,
    $$\EE L_m = (d\EE[\tau_1^2])^m,$$
    due to the gap sizes being identically distributed with $\tau_i\sim \mathrm{Beta}(1, d-1)$. Its second moment is known to be $\frac{2}{d(d+1)}$, hence we find $\EE L_m = \left(\frac{2}{d+1}\right)^m$.
    Now the claim directly follows from Markov's inequality.
\end{proof}

\begin{claim}\label{claim:all_gaps_are_large}
For $0<a<e^{-(d-1)}$
$$\PP(\min_{i=1, \dots, d^{m}} l_i<a^m)\leq\left(2da^{\frac{d-1}{2}}\right)^m$$
\end{claim}

\begin{proof}
The $l_i$ are identically distributed, thus by a union bound, it suffices to prove
\begin{equation}\label{eq:chernoff_goal}
\PP(1/l_1>a^{-m})\leq \left(2a^{\frac{d-1}{2}}\right)^m.
\end{equation}

By definition, $1/l_1$ can be expressed as
$$1/l_1= \prod_{j=1}^{m}A_j^{-1}=\exp\left(-\sum_{j=1}^{m}\log A_j\right),$$
where $A_j\sim \mathrm{Beta}(1, d-1)$. Then $-\log A_j=B_j \sim \mathrm{Exp}(d-1)$, and
$$\PP(1/l>a^{-m})=\PP(\sum_{j=1}^{m} B_j>-m\log a).$$
The moment-generating function of $B_j$ is $\frac{d-1}{d-1-t}$ for $t<d-1$, thus by Chernoff's bound, for any $0<t<d-1$ we have
$$\PP(\sum_{j=1}^{m} B_j>-m\log a)\leq \left(\frac{d-1}{d-1-t}\right)^m \exp(tm\log a)=\left(\frac{d-1}{d-1-t}a^t\right)^m.$$
Substituting $t=\frac{d-1}{2}$, we directly obtain \eqref{eq:chernoff_goal}.    
\end{proof}

We say that $\mu$ is {\it $a$-gapped on the $m$th level} if $\min_{i=1, \dots, d^{m}} l_i\geq a^m$. The previous claim assures that for small enough $a$, depending only on $d$, apart from exponentially small probability $\mu$ is $a$-gapped on the $m$th level. Fix such an $a$ for the rest of the proof.

\begin{claim}\label{claim:expectation_offset_small}
    Fix $s>0$. For some $m\geq m_0$, if $\tau|_{m}=t|_{m}$ is fixed so that $\mu$ is $a$-gapped on the $m$th level, then for $n>a^{-3m}$, 
    $$|\EE [H_n(\mu)] - \EE [H_n(\mu^{t|_m}])|<ns.$$
\end{claim}

\begin{proof}[Proof of Claim \ref{claim:expectation_offset_small}]
    Applying Lemma \ref{lemma:geom_sep} to the natural $m$th level decomposition of $\mu$ sampled after $\tau|_{m}=t|_{m}$ being fixed, and taking expectation, due to self-similarity we find 
    \begin{equation}\label{eq:expectation_offset}
    \EE [H_n(\mu^{t|_m})] = O(d^m\log(n+d^m)) + \sum_{k_1 + \dots + k_{d^m}=n}\PP(E_{k_1, \dots k_{d^m}})\sum_{i=1}^{d^m}\EE [H_{k_i}(\mu)],
    \end{equation}
    where $E_{k_1, \dots, k_{d^m}}$ denotes the event that upon sampling $n$ points uniformly from $[0, 1]$, $k_i$ points is drawn from the $i$th interval of the partition determined by the edge weights up to $\tau|_{m}=t|_{m}$.
    
    Due to the convergence $\EE [H_n(\mu)] /n\to c$, we have $\EE [H_{k}(\mu)] = c(k+\epsilon(k))$, where $\epsilon(k)\to 0$ as $k\to \infty$. Consequently, for any choice of $k_1, \dots, k_{d^m}$, summing to $n$, we have
    $$\left|\EE [H_n(\mu)] -\sum_{i=1}^{d^m}\EE [H_{k_i}(\mu)]\right|=\left|\sum_{i=1}^{d^m}k_i \epsilon(k_i)\right|\leq n\max \epsilon(k_i).$$
    As the error term $O(d^m\log(n+d^m))$ goes to zero after division by $n$ for large $n$, it suffices to show that with high probability, $\max \epsilon(k_i)\to 0$, that is $\min k_i \to \infty$.
    
    To this end, observe that the number $k_i$ of points drawn from the $i$th interval is distributed as $\Binom(n, l_i)$, thus by Hoeffding's inequality,
    $$\sum_{\substack{k_1+\dots+k_{d^m}=n \\ \exists k_i<nl_i/2}} \PP(E_{k_1, \dots, k_{d^m}})\leq d^m \PP\left(k_1<\frac{na^m}{2}\right)\leq d^m\exp\left(-\frac{n}{2}a^{2m}\right)\to 0,$$
    by the choice of $n$, as $m\to \infty$. However, it precisely means that $\max \epsilon(k_i)$ goes to zero apart from probability tending to 0, verifying
    $$|\EE [H_n(\mu)] - \EE [H_n(\mu^{t|_m})]|<s.$$
    for $m$ large enough.
\end{proof}

Now we have every necessary tool to finish the proof of Lemma   \ref{lemma:convergent_probability_sum}.
    Fix $s>0$ and choose $m_0$ to $s$ according to Claim \ref{claim:expectation_offset_small}. For $m\geq m_0$, we will bound the following probability
    $$P_m = \PP(\exists n\in [a^{3m}, a^{3(m+1)}]:\ |H_n(\mu) - \EE [H_n(\mu)]| >2ns\}.$$
    Verifying that $\sum_{m=m_0}^{\infty}P_m<\infty$ is sufficient due to Borel-Cantelli.
    
    We say that $\tau$ is $m$-good if it is $a$-gapped on the $m$th level and $L_m\leq\left(\frac{2}{d+1/2}\right)^m$. We can bound the probability $P_m$ from above by the following sum of probabilities:
    $$P_m\leq \PP(\tau\text{ is not $m$-good})+\sum_{n\in [a^{3m}, a^{3(m+1)}]}\PP(\tau\text{ is $m$-good and }|H_n(\mu) - \EE [H_n(\mu)]| >2ns)$$
    Claims \ref{claim:squared_sum_small}-\ref{claim:all_gaps_are_large} establish that the first term summed for $m$ is convergent. On the other hand, by the natural disintegration given by \eqref{eq:disintegration}
    $$\PP(\tau\text{ is $m$-good and }|H_n(\mu) - \EE [H_n(\mu)]| >2ns)=\int_{t\text{ is $m$-good}}\PP(|H_n(\mu^{t|_m}) - \EE [H_n(\mu)]| >2ns)d\tau(t).$$
    In this set, by Claim \ref{claim:expectation_offset_small} $|\EE [H_n(\mu)] - \EE [H_n(\mu^{t|_m})]|<ns$, thus the integrand is bounded by 
    $$\PP(|H_n(\mu^{t|_m}) - \EE [H_n(\mu^{t|_m})]| >2ns).$$
    For this probability, we can apply Hoeffding's inequality in a similar fashion as in the proof of Claim \ref{claim:large_dev_principle}. Formally, the only thing that changes is that upon bounding
    $$\PP\left(\left|\sum_{i=1}^{d^m}W_{i, n, m} - \EE\left(\sum_{i=1}^{d^m}W_{i, n, m}\right)\right|\geq s/2\right),$$
    the squared sum of the length of ranges of these random variables is not $d^{-m}$, but bounded from above by $\left(\frac{2}{d+1/2}\right)^m$, implying the upper bound
    $$2\exp\left(-\frac{s^2(d+1/2)^m}{2^{m+1}K^2}\right).$$
    This upper bound has a superexponential decay in $m$, thus summing this for $n\in [a^{3m}, a^{3(m+1)}]$ still results in a convergent series. Thus  $\sum_{m=m_0}^{\infty}P_m<\infty$ indeed, concluding the proof.
\end{proof}

\begin{proof}[Proof of Theorem \ref{thm:random_automorphism2}]
We have proved the convergence of the means (Lemma \ref{lemma:random_aut_conv_means}), discussed the sign of the limit (Lemma \ref{lemma:positive_limit}), and verified that almost surely $H_n(\mu)/n$ differs from its mean by more than $s$ finitely many times for any $s>0$ (Lemma \ref{lemma:convergent_probability_sum}). Hence $H_n(\mu)/n$ converges almost surely to a positive number, save for the proposed trivial degenerate cases of a vanishing limit.
\end{proof}

\section*{Acknowledgments}
The author is thankful to Miklós Abért for proposing the idea to start studying sampling entropy sequences and for asking great questions during our numerous discussions about the topic. Besides that, the author is thankful to Balázs Ráth, Boglárka Gehér, Aranka Hrušková, and Ágnes Cs. Kúsz for the helpful discussions. The REU group consisting of Péter Fazekas, Bowen Li, Aayan Pathan, Balázs Szepesi, and Sára Szepessy supervised by the author in this topic also deserves gratitude for the motivating meetings.

\printbibliography

\end{document}